%% file: main_3.tex
\documentclass[11pt]{article}
\usepackage{amsfonts,latexsym,amsthm,amsmath,amscd,geometry}
\usepackage{graphicx}
\usepackage{amssymb}
\usepackage{latexsym}
\usepackage{xcolor}
\usepackage{indentfirst}
\makeatletter
\@addtoreset{equation}{section}
\makeatother
\input{newmacros}

\newcommand \nc{\newcommand}
\theoremstyle{plain}
\newtheorem{theorem}{Theorem}[section]
\newtheorem{lemma}[theorem]{Lemma}
\newtheorem{proposition}[theorem]{Proposition}

\newtheorem{remark}[theorem]{Remark}

\nc{\ba}{\begin{array}}\nc{\ea}{\end{array}}
\nc{\be}{\begin{eqnarray}}\nc{\ee}{\end{eqnarray}}
\nc{\beq}{\begin{equation}}\nc{\eeq}{\end{equation}}
\nc{\bex}{\begin{eqnarray*}}\nc{\eex}{\end{eqnarray*}}
\nc{\btm}{\begin{theorem}} \nc{\etm}{\end{theorem}}
\nc{\blm}{\begin{lemma}} \nc{\elm}{\end{lemma}}
\nc{\R}{\mathbb{R}}  

\newcommand{\subjclass}[1]{\par\medskip\noindent 2020 \textit{Mathematics Subject Classification}. #1.}
\newcommand{\keywords}[1]{\par\medskip\noindent \textit{Key words and phrases}. #1.}

\author{}
\date{}

\begin{document}
\title{The $L^2$ contraction of solutions with large perturbation in multiple space dimensions from the oscillatory dispersive planar shock}

\author{Geng Chen
\footnote{Department of Mathematics, University of Kansas, Lawrence, KS 66045, U.S.A. Email: gengchen@ku.edu},
\quad
Harish Ravichandar
\footnote{Department of Mathematics, University of Kansas, Lawrence, KS 66045, U.S.A. Email: h415r996@ku.edu},
\quad
Yannan Shen
\footnote{Department of Mathematics, University of Kansas, Lawrence, KS 66045, U.S.A. Email: yshen@ku.edu}
}

\maketitle
\begin{abstract}
In this paper, we show the \(L^2\)-contraction property of the planar oscillatory or monotone dispersive shock profiles of the dissipative Kadomtsev-Petviashvili (KP) equation modelling water waves and the multi-dimensional Korteweg–de Vries (KdV) Burgers equation under arbitrarily large perturbations in two space dimensions, up to Lipschitz time-dependent shifts. This stability result extends the results of two recent papers by Chen, Eun, Kang, and Shen on \(L^2\)-contraction for the KdV-Burgers equation.


\end{abstract}
\section{Introduction}
In this paper, we consider two models which can be used to describe the dispersive shock wave: dissipative Kadomtsev-Petviashvili  equation and the multi-dimensional Korteweg–de Vries (KdV)–Burgers equation, also known as the Zakharov–Kuznetsov–Burgers (ZKB) equation.

The main goal of this paper is to study the stability of the planar dispersive shock wave under the perturbation in multiple space dimensions.
\paragraph{\bf Dissipative KP type equation:}
\setcounter{equation}{0}
We first consider a dissipative Kadomtsev-Petviashvili type equation  in two space dimensions (2D):
\begin{equation}\label{KP FIrst equation}
    (u_{t} + uu_x -\varepsilon_1 u_{xx}-\varepsilon_2 u_{yy}+\delta u_{xxx})_x+\lambda u_{yy}=0,
\end{equation}     
or equivalently
\begin{equation}
\label{KP main equation}
    \begin{cases}
    u_{t} + uu_x -\varepsilon_1 u_{xx}-\varepsilon_2 u_{yy}+\delta u_{xxx}=v \\
    v_x=-\lambda u_{yy},  
        \end{cases}
\end{equation} 
with initial data
\begin{equation}\label{ID1} 
    u(0,x,y)=u^0(x,y), \qquad u^0(x,y) = u^0(x,y+1), 
    \end{equation}
where \(u = u(t,x,y)\in \mathbb{R}\) and \(v = v(t,x,y)\in \mathbb{R}\) are unknowns defined on \(t > 0, \ x \in \mathbb{R}\), \(y\in \mathbb{T}^{1}\) with \(\mathbb{T}^{1}\) :=\(\mathbb{R} / \mathbb{Z}\), denoting the one-dimensional flat torus, while \(\varepsilon_1>0,\ \varepsilon_2>0,\ \delta>0\), and \(\lambda=\pm1\) are constants. For simplicity, we always use the notation $$\Omega := \mathbb{R}\times \mathbb{T}^{ 1}.$$ 

When \(\varepsilon_1=\varepsilon_2 = 0\), equation \eqref{KP FIrst equation} reduces to the standard (non-dissipative) KP equation, which is referred to as KP-I when \(\lambda = -1\) and KP-II when \(\lambda = 1\). 

The KP equation
is a prototypical universal model for weakly two-dimensional,
weakly nonlinear waves \cite{kadomtsev1970stability}. This equation is used to model shallow water; for example, see \cite{ablowitz1981solitons}, \cite{kodama2018solitons}. We also refer the reader to the recent paper \cite{biondini2025mach}, which studies the Mach reflection problem for the KP equation in the KP-II case \((\lambda>0)\), as well as to the references therein for further background on the KP equation. See other related works \cite{feng2002stable,fukuda2026asymptotic,molinet2000cauchy} and the references in \cite{biondini2025mach}, where, in particular, some dissipative KP equations are considered.

The equation \eqref{KP FIrst equation} can be viewed as a dissipative version of KP equation, where $\varepsilon_1$ and $\varepsilon_2$ are two positive dissipative constants. Physically, $\varepsilon_2$ is much smaller than $\varepsilon_1$.  Especially, this model can be used to describe the 2D viscous-dispersive shock wave. Without confusion, from now on, we always use KP equation to denote \eqref{KP FIrst equation}. See Remark \ref{remark2} for some further discussions on the model.

The meaning of the variable $u$ depends on the physical context considered, while $x$ represents the main direction of propagation, $y$ represents a transverse spatial variable and $t$ the time variable. In shallow water, $u$ quantifies the deviation of the water surface from its equilibrium value.

\paragraph{\bf ZKB equation:}
We also consider the following multi-dimensional Korteweg–de Vries (KdV)–Burgers equation, also known as the Zakharov–Kuznetsov–Burgers (ZKB) equation.
\begin{eqnarray}\label{KdVzk main equation0}
    u_{t} + u u_{x_1} - \varepsilon \Delta u + \delta u_{x_1x_1x_1} + \kappa \sum_{i=2}^n u_{x_1x_ix_i} = 0, \qquad u(0,x) = u_{0}(x), \nonumber \\
    u(x_1, \ldots,x_i, \ldots,x_n)=u(x_1, \ldots,x_i+1, \ldots,x_n) \quad \forall i \in \{2,3,...,n\}. 
\end{eqnarray}
where \(u = u(t,x) \in \mathbb{R}\) is the unknown function, defined for \(t > 0\) and
\(x = (x_{1}, x')\) with \(x_{1} \in \mathbb{R}\) and
\(x' \in \mathbb{T}^{n-1} := \mathbb{R}^{n-1} / \mathbb{Z}^{n-1}\), \(n \geq 2\) being the \(n-1\) dimensional flat torus.

The equation has been studied in superthermal electron–positron–ion plasmas. For the derivation of the model when \(n=3\), see \cite{el2011zakharov}. This equation has also been investigated in the context of wave turbulence theory; see \cite{ma2022almost}. When $\kappa$=0, the equation reduces to a multidimensional KdV-type equation; see \cite{carvajal2017well}.

\bigskip

In this paper, we establish the \(L^2\)-contraction and stability of 2D large perturbations around the planar dissipative shock for both \(\lambda = -1\) and \(\lambda = 1\), corresponding to the KP-I and KP-II cases, respectively, and the ZKB equation. 

In particular, when the solution $u$ depends only on $x$ (or $x_1$), i.e. in the case of planar waves, both the KP and ZKB equations reduce to the Korteweg-De Vries (KdV)-Burgers equation
\beq\label{kdvb}
  u_{t} + uu_x -\varepsilon_1 u_{xx}+\delta u_{xxx}=0,
  \eeq
  where $\varepsilon_1=\varepsilon$, $x=x_1$ for the ZKB equation.

Let's first consider a planar traveling viscous-dispersive shock wave solution, along the $x$ direction: \(\tilde{u} (\xi) = \tilde{u} (x - \sigma t)\) for (\ref{KP main equation}) satisfying
\begin{equation}
\label{shock equation}
\begin{cases}
    -\sigma \tilde{u}' + (\frac{\tilde{u}^2}{2})^{\prime} = \varepsilon_1 \tilde{u}^{\prime \prime} - \delta\tilde{u}^{\prime \prime \prime} , \\
    \tilde{u} (\xi)\to u_{\pm}\mathrm{~as~}\xi \to \pm \infty ,\\
    \tilde{u}^\prime(\xi) \to 0\ \mathrm{as}\ \xi \to \pm  \infty.
    \end{cases}
\end{equation}
Here the speed of shock \(\sigma\) is determined by the Rankine-Hugoniot condition
\[\sigma = \frac{u_- + u_+}{2},\]
and the constants \(u_{\pm}\) satisfy the Lax entropy condition \(u_{- } > u_{+}\). 
This traveling wave solution $\tilde{u}$ is indeed a solution for the KdV-Burgers equation \eqref{kdvb}.

In this paper, we consider the stability of large perturbations and $L^2$-contraction up to a Lipschitz time-dependent shift for the 2D solution $u$ of \eqref{KP FIrst equation} and multi-D solutions of \eqref{KdVzk main equation0}, around the viscous-dispersive planar shock wave in \eqref{shock equation}. More precisely, we study the stability of the solution $u$, satisfying
$\lim_{x\rightarrow \pm \infty}u(t,x,y)=u_\pm,$ for \eqref{KP main equation} and \eqref{KdVzk main equation0} around the traveling wave solution $\tilde{u}$ for \eqref{shock equation}, under any large $L^2$ perturbation.

 The behaviors of the viscous-dispersive shock wave described by \(\util\) for \eqref{shock equation} depend on the critical parameter \( \frac{\d(u_--u_+)}{2\e_1^2}\).
In fact, when \( \frac{\d(u_--u_+)}{2\e_1^2}\leq \frac{1}{4}\), i.e., when the viscosity effect dominates the dispersion effect, the linearized equation of \eqref{shock equation} around the equilibrium state \(\util=u_-\) has real eigenvalue(s).
So the traveling wave solution \(\util\) is monotonically decreasing.
The much more interesting case is when the dispersion effect dominates, i.e., when \( \frac{\d(u_--u_+)}{2\e_1^2}> \frac{1}{4}\).
In this case, the non-monotone traveling viscous-dispersive shock wave solutions \(\util\) oscillate around \(\util=u_-\) infinitely many times as \(\x \to -\infty\). See the classical paper \cite{bona} or \cite{chen2026mono,chen2026uniform}. 

The stability of oscillatory viscous-dispersive shock for the KdV-Burgers equation is an important and interesting problem. In a very recent paper, \cite{chen2026uniform}, the $L^2$-contraction and stability theory has been established when 
\(  \frac{1}{4}<\frac{\d(u_--u_+)}{2\e_1^2}<\frac{1}{2}.\)
This result relies on the sharp estimate on the traveling wave solution, which will be introduced in Section 2. This will also serve as our basis to study the stability of 2D perturbation. Note the upper bound $\frac{1}{2}$ for $\frac{\d(u_--u_+)}{2\e_1^2}$ is not optimal for our method, as for the KdV-Burgers equation. Here we still use it to make the argument much simpler, and focus on how to show the $L^2$-contraction for the non-monotonic case.

One can find another interesting paper \cite{Hur} giving a sufficient condition on the $L^2$-contraction for KdV-Burgers equation, together with a computer-assisted analysis on this condition for the non-monotonic case. Especially, the upper bound of $\frac{\d(u_--u_+)}{2\e_1^2}$ considered in \cite{Hur} is larger than $\frac{1}{2}$.

The method in \cite{chen2026uniform} also gives a new proof for the $L^2$-contraction for the monotonic case when \(0<\frac{\d(u_--u_+)}{2\e_1^2}\leq \frac{1}{4}\), see \cite{chen2026mono}, and also \cite{Hur} for an earlier proof. 

The goal of the current paper is to extend the theory in \cite{chen2026mono,chen2026uniform} to multi-D solutions.

We now state the main results of this paper.
Since the results for the ZKB equation are very similar to those for the KP equation, we defer them to Section 7 and present only the KP results here.

We introduce the following function space:
\[
X_T :=
\left\{
u : \mathbb{R}_+ \times \Omega \to \mathbb{R}
\;\middle|\;
\begin{aligned}
&u - f \in C([0,T];H^1(\Omega)) \cap L^2(0,T;H^2(\Omega)), \\
&\int_{\mathbb{R}} (u(t,x,y)-f(x)) dx=0,
\quad \forall\, t \in [0,T], \ \forall\, y \in \mathbb{T}^1
\end{aligned}
\right\},
\]
where $f$ is a smooth monotone profile satisfying
\[
f(x) :=
\begin{cases}
u_-, & x \leq -1, \\
u_+, & x \geq 1.
\end{cases}
\]

First, we state a global existence result, where the proof is standard. To the best of our knowledge, there is no existence result covering the case where $\lim_{x\rightarrow -\infty}u(t,x,y)$ and $\lim_{x\rightarrow +\infty}u(t,x,y)$ take different values.
\begin{theorem}\label{gloabal existence}
Assume that the initial data $u^0$ satisfies the condition \beq\label{meanzero}
\int_{\mathbb R} (u^0(x,y)-f(x))dx=0.
\eeq
and
\[
u^0 - f \in H^s(\Omega)
\]
for any $s>1$, where $f$ is a smooth monotone profile satisfying
\[
f(x) :=
\begin{cases}
u_-, & x \leq -1, \\
u_+, & x \geq 1.
\end{cases}
\]
Then, for any $T>0$, there exists a unique solution to \eqref{KP main equation} and \eqref{ID1} satisfying
\[
u \in X_T.
\]
\end{theorem}

\begin{remark}
We note that the zero mean condition
\beq\label{zeromean2}
\int_{\mathbb{R}} (u(t,x,y)-f(x))\,dx = 0
\eeq
is naturally encoded in the KP equation. And the initial zero mean condition \eqref{meanzero} is a consistency condition.

Indeed, integrating \eqref{KP FIrst equation} with respect to \(x\) over \(\mathbb{R}\), we obtain
\[
\partial_{yy}\left(\int_{\mathbb{R}} u(t,x,y)\,dx\right)=0.
\]
Since the solution is periodic in the \(y\)-variable on \(\mathbb{T}^1\),  $\int_{\mathbb{R}} u(t,x,y)\,dx$ is independent of $y$.
At $t=0$, this gives the initial zero mean condition \eqref{meanzero}, given that $\lim_{x\rightarrow \pm \infty}f(x)=u_\pm$. We assume slightly more on $f(x)$ due to some technical reason.
Then, Theorem \ref{gloabal existence} guarantees that the solution always satisfies the zero mean condition \eqref{zeromean2}.
 Consequently, the zero-mean condition is not an additional assumption, but rather an intrinsic property of the KP equation.
 
For Theorem \ref{gloabal existence}, we need to assume that the initial data are in $H^s(\Omega$), i.e. $u^0-f\in H^s(\Omega)$ with any $s>1$, instead of $s=1$, due to the application of Sobolev embedding theorem.
\end{remark}

Now we can state our main stability theorem for both the monotone and non-monotone cases.

\begin{theorem}[Monotone Case] \label{Main theorem for monotone}
    We denote \(u_{-}- u_+ = 2 s>0\). Let \(\varepsilon_1 >0 \) and \(\delta>0\) be constants.
 Assume that $ \frac{\d(u_--u_+)}{2\e_1^2}\leq \frac{1}{4} $, so that the traveling wave solution \(\tilde{u}\) of \eqref{shock equation}  is a monotonically decreasing solution connecting $u_{-}\ to\ u_{+}$.    Then the following inequality holds:
    \begin{equation} \label{main lemma0}
  \lambda_1\ (s - \tilde{u}(\xi))\ (\tilde{u}(\xi) + s) \le -\varepsilon_1 \tilde{u}'(\xi), \qquad \forall \xi \in \mathbb{R} 
\end{equation}
for any   \(\lambda_1 < \sqrt{2}-1.
    \)
    
 Let \(u^{0}\) be given initial data satisfying \(u^{0}\in  H^s(\Omega)\) for $s>1$, and let \(u\in X_T\) be a solution to equation \eqref{KP main equation} with initial data $u^{0}$ satisfying \eqref{ID1} and \eqref{meanzero} . If \beq\label{mono_ass} s < 4 \pi \sqrt{\e_1\e_2}, \qquad \eeq then, for any $\frac{1}{4}<\lambda_1<\sqrt{2}-1$, there exists a Lipschitz shift \(X(t)\) such that
\begin{equation}
\label{theorem1}
\begin{aligned}
\int_{\Omega} \big| u(t,x,y) - \tilde{u}(x - \sigma t - X(t)) \big|^2 \, dx\,dy 
+ C_{*}\varepsilon_1 \int_0^T \int_{\Omega} \left (\left( u(t,x,y)-\tilde{u}(x - \sigma t - X(t)) \right)_x \right)^2 \, dx\,dy \,dt \\
+ C^{*}  \int_0^T\int_{\Omega} \left (\left(u(t,x,y)- \tilde{u}(x - \sigma t - X(t)) \right)_y \right)^2 \, dx\,dy \, dt
\leq \int_{\Omega} |u^0 - \tilde{u}|^2 \, dx\,dy,
\end{aligned}
\end{equation}
where \(C_*= 2-\frac{1}{2\lambda_1}\), \(C^*=2\varepsilon_2-\frac{s^2}{8\pi^2\varepsilon_1} \),  and the Lipschitz shift satisfies the ODE,
\begin{equation} 
    \label{lipstich}
\begin{cases}
{\dot{X} = -\frac{M}{2s}\int_{\Omega}\Big( u(t,x+ X(t),y) - \tilde{u} (x - \sigma t )\Big)\tilde{u}'(x - \sigma t)\, d x\, d y,}\\ {X(0) = 0,}
\end{cases} 
\end{equation}  
with $M=\frac{1}{2}$. Moreover, for $2<p<\infty$, the following holds:
\begin{equation} \label{tstabilitymonotone}
\int_{\Omega}
\left|u(t,x,y)-\tilde{u}(x-\sigma t-X(t))\right|^p \, dx\,dy 
\to 0 \quad \text{as } t \to \infty,
\end{equation}
which establishes time-asymptotic stability.
Furthermore, the shift function satisfies the following time asymptotic behavior
\begin{equation} \label{shiftasy}
|\dot{X}(t)| \to 0 \quad \text{as } t \to \infty, \quad \text{and} \quad \frac{X(t)}{t} \to 0 \quad \text{as } t\to\infty. 
\end{equation} 
\end{theorem}

\begin{theorem}[Exponential decay for monotone traveling wave]\label{exponential decay of monotne}
 For any \(\gamma>\frac{1}{2}\), assume \( \e_1 >\sqrt{\frac{4\gamma^2\delta s}{2\gamma-1}},\)  where we denote \(u_{- } - u_+ = 2 s>0\). 
In this case, the condition $ \frac{\d(u_--u_+)}{2\e_1^2}\leq \frac{1}{4} $ is always satisfied, so $\tilde{u}$ is monotonically decreasing.    
In addition to \eqref{main lemma0}, we also have the following inequality
    \begin{equation} \label{main lemma2}
-\e_1\tilde{u}'(\xi) \le \gamma\ (s - \tilde{u}(\xi))\ (\tilde{u}(\xi) + s), \qquad \forall \xi \in \mathbb{R}. 
\end{equation}
Moreover, let \(\util(0)=\frac{u_-+u_+}{2}\).
Then,
\[
\frac{(u_--u_+)}{2} e^{-\frac{\gamma(u_--u_+)\abs{x}}{\e_1}}
\le u_--\util(x)
\le \frac{(u_--u_+)}{2} e^{-\frac{\lambda_1(u_--u_+)\abs{x}}{2\e_1}}, \qquad \forall x\le0,
\]
\[
\frac{(u_--u_+)}{2} e^{-\frac{\gamma(u_--u_+)\abs{x}}{\e_1}}
\le \util(x)-u_+
\le \frac{(u_--u_+)}{2} e^{-\frac{\lambda_1(u_--u_+)\abs{x}}{2\e_1}}, \qquad \forall x\ge0,
\]
\[
\frac{\lambda_1(u_--u_+)^2}{4\e_1} e^{-\frac{\gamma(u_--u_+)\abs{x}}{\e_1}}
\le -\util'(x)
\le \frac{\gamma(u_--u_+)^2}{2\e_1} e^{-\frac{\lambda_1(u_--u_+)\abs{x}}{2\e_1}}, \qquad \forall x\in\RR,
\]
where \(\lambda_1=\sqrt{2}-1\).
\end{theorem}

The proof of Theorem \ref{exponential decay of monotne} follows arguments very similar to those used in the proof of Theorem 1.4 in \cite{chen2026mono} and is therefore omitted for brevity.

For the non-monotone case, we obtain the $L^2$-contraction when 
\begin{equation} \label{cond1}
\frac{1}{4} < \frac{\d(u_--u_+)}{2\e_1^2} < \frac{1}{2}.
\end{equation}
Without loss of generality, for this case, we always assume $\e_1=1$. 

\begin{remark}
    Due to the Galilean invariant transformation, without loss of generality, we only consider the case when $u_-=-u_+=s$, so $\sigma=0$. See \cite{chen2026uniform}.
\end{remark}

\begin{theorem}[Non-monotone Case] \label{main theorem for non monotone}
       Let \(\e_1=1\) and \(\d\) be a positive constant, such that
\begin{equation} \label{cond1_2}
\frac{1}{4} < \d s < \frac{1}{2},
\end{equation}
and where \(u_\pm=\mp s\) be the far-field states with $s>0$.

 Let $u^0$ be given initial data satisfying  \(u^{0}\in  H^s(\Omega)\) for $s>1$, and let \(u\in X_T\) be a solution to equation \eqref{KP main equation} with initial condition $u^{0}$ satisfying \eqref{ID1} and \eqref{meanzero}. Then, for \beq\label{s_def} s < \sqrt{\frac{7.86\ \e_2}{M}} \pi ,\qquad\hbox{with}\qquad M = \frac{4}{3},\eeq
 the following $L^2$-contraction holds:
\begin{equation}\label{theorem2}
\begin{aligned}
\int_{\Omega} \big| u(t,x,y) - \tilde{u}(x - \sigma t - X(t)) \big|^2 \, dx\,dy 
+ \frac{1}{5} \int_0^T \int_{\Omega}\left(\left( u(t,x,y)-\tilde{u}(x - \sigma t - X(t)) \right)_x \right)^2 \, dx\,dy \,dt \\
+ C^{**}  \int_0^T\int_{\Omega} \left (\left( u(t,x,y)-\tilde{u}(x - \sigma t - X(t))\right)_y \right)^2 \, dx\,dy \, dt
\leq \int_{\Omega} |u^0 - \tilde{u}|^2 \, dx\,dy, \quad 
\end{aligned}
\end{equation}
where \(C^{**}=2\e_2-\frac{2.034 }{8\pi^2}  Ms^2  \geq 0\), and the shift function $X(t)$ satisfies $X(0) = 0$,
\begin{equation}\label{shift non monotone}
\dot{X}(t) = -\frac{M}{2s} \int_{\Omega} \left( u(t, x + X(t),y) - \tilde{u}(x - \sigma t) \right) \tilde{u}'(x - \sigma t)\, dx\, dy.
\end{equation}
Moreover, for $2<p<\infty$, the following holds:
\begin{equation} \label{tstabiltynonmonotone}
\int_{\Omega}
\left|u(t,x,y)-\tilde{u}(x-\sigma t-X(t))\right|^p \, dx\,dy 
\to 0 \quad \text{as } t \to \infty,
\end{equation}
which shows the time-asymptotic stability. Furthermore, the following time asymptotic behavior of the shift function holds:
\begin{equation} \label{shiftasy2}
|\dot{X}(t)| \to 0 \quad \text{as } t \to \infty ,\quad and \quad \frac{X(t)}{t} \to 0 \quad \text{as } t\to\infty.
\end{equation} 
\end{theorem}

Note the condition in \eqref{s_def} involves $\e_1$ if we do not assume $\e_1=1$.

\begin{remark}
(1)  Since \( u \in X_T \), the existence, uniqueness, and Lipschitz continuity of the shift function \( X(t) \) in Theorems~\ref{Main theorem for monotone} and~\ref{main theorem for non monotone} can be both established using Lemma~C.1 in~\cite{KV-JDE22}. \\
(2)  Note that \begin{equation} 
|\dot{X}(t)| \to 0 \quad \text{as } t \to \infty \quad \implies \quad \frac{X(t)}{t} \to 0 \quad \text{as } t\to\infty.
\end{equation} 
Specifically, the shift increases at most sublinearly over time, and the shifted shock wave asymptotically preserves the shape of the original traveling wave profile.

\end{remark}

\begin{remark} \label{remark2}
Physically, to derive the KP equation, we need to assume that $0 < \varepsilon_2 \ll \varepsilon_1$. 
This scaling captures environments where the dominant dissipative mechanisms—such as fluid viscosity in a narrow channel or collisional damping in a directional plasma flow—act primarily along the main axis of wave propagation ($x$), while transverse diffusive effects ($y$) are orders of magnitude weaker. 

However, from a strictly functional analytic perspective, taking the singular limit $\varepsilon_2 \to 0$ falls outside the current proof structure. A strictly positive transverse dissipation ($\varepsilon_2 > 0$) is indispensable in our method, where the reader can find the condition on $\varepsilon_2$ in the theorems, for example, \eqref{s_def} for the non-monotonic case. More precisely, by the scaling, \eqref{s_def} shows that the shock strength $2s\leq O(\sqrt{\varepsilon_2})$ when $\varepsilon_1$ is set to be $1$.  This condition supplies the necessary parabolic regularization in the $y$-direction to control transverse gradients. In this sense, our $L^2$-contraction and stability result holds when the travelling wave includes an oscillating small shock.

On the other hand, the ZKB equation has a uniform dissipation constant $\varepsilon$ in all directions. Our results on this equation, which are similar to those for the KP equation, can be found in Section 7.
\end{remark}

This paper is organized into seven sections. In Section 2, we review the properties of planar traveling shock wave in \cite{chen2026uniform}. Sections 3-6 focus on the KP equation.
In Section 3, we give some preliminary estimates that will be used later. In Section 4, we prove the main theorems for the monotone case. In Section 5, we prove the main theorem for the non-monotone case. The global existence theorem is proved in Section 6. We give the results for ZKB equation in Section 7.

\section{A review of the planar traveling wave}
\setcounter{equation}{0}
In this section, we review the properties of the planar traveling wave $\tilde{u}$ for \eqref{shock equation}, especially the oscillatory case, given in \cite{chen2026uniform}. Recall that $\tilde{u}$ is also a solution of the KdV--Burgers equation.

We integrate the traveling wave equation \eqref{shock equation} over \((\x,\pm\infty)\) to find that
\begin{equation} \label{t00}
-\s (\util-u_\pm) + \frac{1}{2}\util^2-\frac{1}{2}u_\pm^2 =\e_1\util' -\d\util''.
\end{equation}
The qualitative behavior of the shock profile can be understood from the associated phase portrait, shown in Figure~\ref{pp}.

\begin{figure}[htp] \centering
    \includegraphics[width=.45\textwidth]{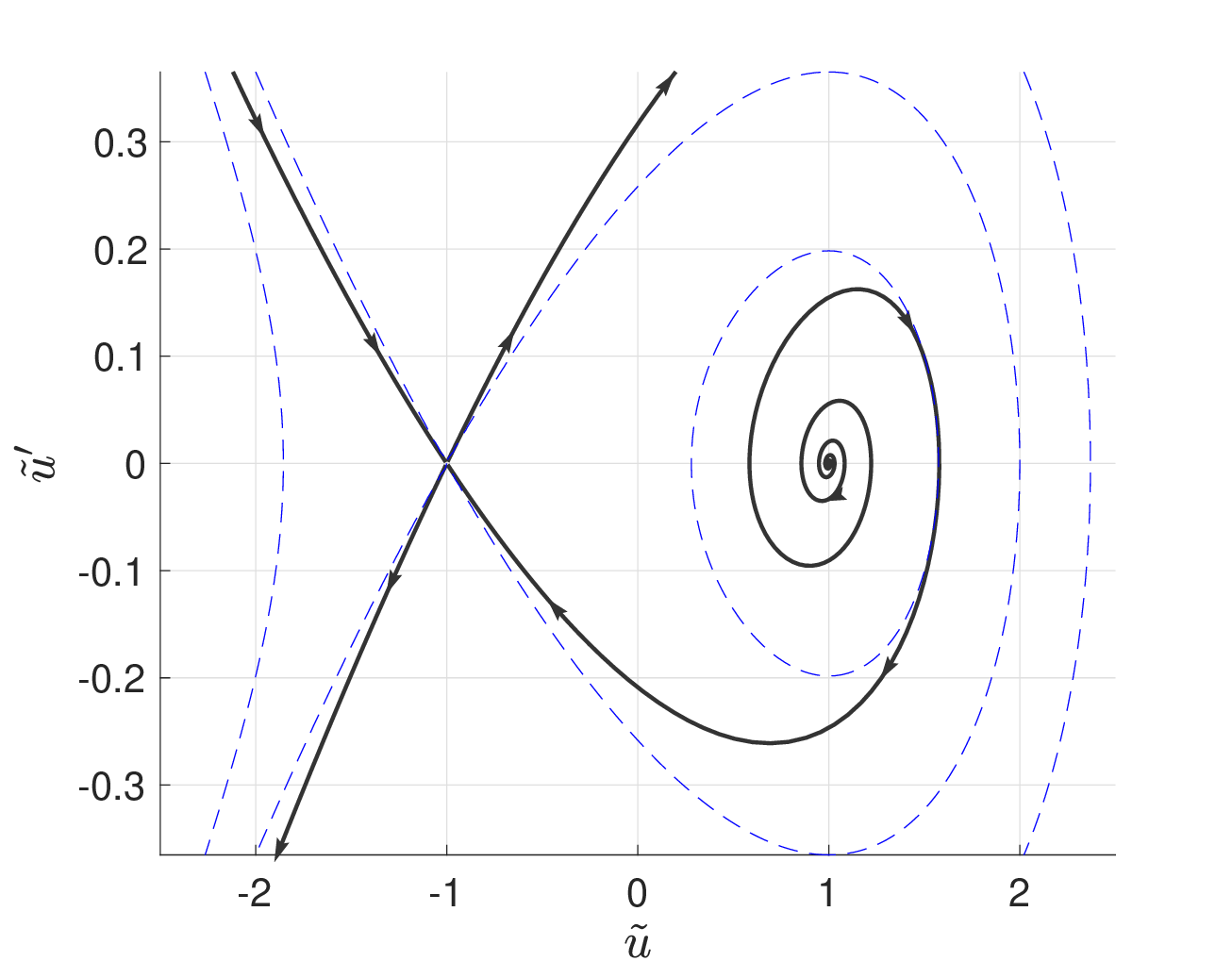}
    \includegraphics[width=.45\textwidth]{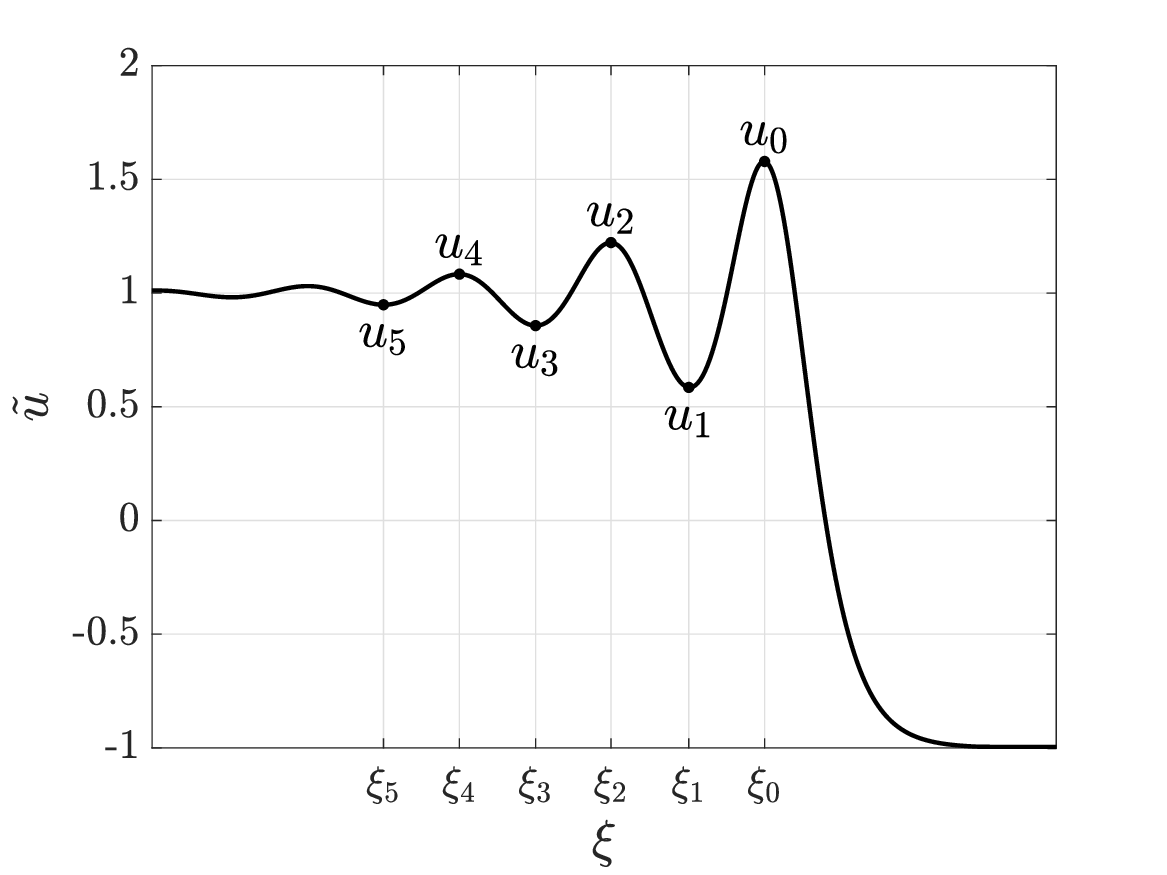}
    \caption{The left panel displays the phase portrait of solutions to \eqref{t00} with \(\d=10\),  \(\e_1=1\) and \(u_{\pm}=\mp 1\) in the \((\util,\util')\)-plane.
    The blue dotted curves correspond to solitary wave solutions of KdV-Burgers equation with \(\e_1=0\), while the black solid curve represents the heteroclinic orbit for \(\e_1=1\) and \(u_--u_+=2\).
    This orbit spirals from \(\util=u_-\) to \(\util=u_+\) and corresponds to the shock profile \(\util(\x)\) shown in the right panel.\label{pp}}
\end{figure}

As shown in the left panel of Figure~\ref{pp}, in the \((\util,\util')\)-plane, the heteroclinic orbit connects the two steady states \((\util=u_\pm, \util'=0)\) and corresponds to the dispersive shock solution shown in the right panel. Here we simply refer to it as a dispersive shock instead of a viscous-dispersive shock, for brevity.

If we linearize \eqref{t00} around the equilibrium \((\util=u_-, \util'=0)\), the eigenvalues are given by 
\begin{equation} \label{L-slope}
\l_{u_-} = \frac{\e_1 \pm \sqrt{\e_1^2-2\d(u_--u_+)}}{2\d}.
\end{equation}
When \(\e_1>0\) and \(\frac{\d(u_--u_+)}{2\e_1^2} >\frac{1}{4}\), the dispersion dominates the dissipation, so the dispersive shock is non-monotonic as in the right panel of Figure~\ref{pp}.
On the other hand, if \(0<\frac{\d(u_--u_+)}{2\e_1^2} \leq\frac{1}{4}\), the dispersive shock is monotonically decreasing.

Now we introduce our result for the traveling wave solution $\tilde{u}$ when  \(\frac{\d(u_--u_+)}{2\e_1^2} >\frac{1}{4}\), i.e., when the dispersive shock is non-monotonic.
We use the following notation: we index the local extrema of \(\util(\x)\) from right to left.
The rightmost extremum is denoted by \(u_0\), and the remaining extrema are indexed sequentially as \(u_1, \ u_2, \cdots\) in the order they appear when moving leftward.
We write \(\x_i\) for the spatial location at which the value \(u_i\) is attained, so that \(\util(\x_i)=u_i\), as shown in Figure \ref{pp}.

The results in \cite{chen2026uniform} for traveling wave solutions  of \eqref{shock equation} are summarized in the following theorem.

\begin{theorem}[{\cite[Theorem 2.1]{chen2026uniform}}] \label{thm:shock}
Let \(\e_1,\d>0\) be constants and let \(u_\pm\) be given states.
Assume that \(u_->u_+\) and 
\begin{equation} \label{cond2}
\frac{1}{4} < \frac{\d(u_--u_+)}{2\e_1^2} < A \le 1.
\end{equation}
Then, the following holds: 
\begin{equation} \label{u0-upper}
u_0-u_- \le \frac{(u_--u_+)}{2} \frac{1}{100A} \Big(\sqrt{9+25(\sqrt{1+4A}-1)^2}-3\Big)^2
\end{equation}
which provides an upper bound for the rightmost local maximum of the dispersive shock profile \(\util\) of  \eqref{shock equation}.
Moreover, the amplitudes of the oscillations \(\abs{u_i-u_-}\) decay exponentially in the sense that for any odd positive integer \(i>0\) and for each interval of $\xi$ with increasing $\util$, it holds that
\begin{equation} \label{inc-decay}
\frac{u_{i-1}-u_-}{u_--u_i} \ge \r_*>1,
\end{equation}
and for each interval of $\xi$ with decreasing $\util$, it holds that
\begin{equation} \label{dec-decay}
\frac{u_--u_i}{u_{i+1}-u_-} \ge \r^*>1,
\end{equation}
where the decay rates for selected values of \(A\) are given by:

\begin{table}[h]
\centering
\begin{tabular}{c c c c c}
\hline
$A$ & $\r_*$   & $\r^*$\\
\hline
$1/3$ & 6.05 & 6.14\\
$1/2$ & 4.64 & 4.77\\
$2/3$ & 3.89 & 4.06\\
$3/4$ & 3.63 & 3.81\\
$1$ & 3.10 & 3.30\\
\hline
\end{tabular}
\caption{Decay Rates}
\label{tab:decay}
\end{table}
\end{theorem}

\section{Preliminaries}
\setcounter{equation}{0}
 We change the variable from $x$ to \(\xi = x - \sigma t\) and rewrite (\ref{KP main equation}) in terms of $u(t,\xi,y)$ as
\begin{equation}
\label{change of varibale equation}
\begin{cases}
    u_{t} - \sigma u_{\xi} + uu_{\xi} = \varepsilon_1 u_{\xi\xi}+ \varepsilon_2 u_{yy} -\delta  u_{\xi \xi \xi} + v,  \\
    v_\xi = - \lambda u_{yy}.
\end{cases}
\end{equation}
Here we abuse the notation and still use $u$ to denote the function  $u(t,\xi,y)$.
\begin{lemma}[{\cite[Lemma 2.9]{kang2020contraction}}]\label{KV} 
For any $w:[a,b] \to \mathbb{R}$ satisfying \[\int_a^b(y-a)(b-y)|w_y|^2 \,dy < \infty,\] 
we have,
\begin{equation}
    \label{poincare0}\int_a^b w^2\,dy\leq \frac{1}{2}\int_a^b(y-a)(b-y)|w_y|^2 \,dy+\frac{1}{b-a}(\int_a^b w \,dy)^2.
\end{equation}
\end{lemma}

\begin{lemma}\label{d_dtlemma}
    Let $\tilde{u}$ be a planar viscous-dispersive shock satisfying equation (\ref{shock equation}), and for any $T>0$, let $u\in X_{T}$ be a solution of the equation \eqref{KP main equation}. Then for any Lipschitz curve $X : [0,T] \to \mathbb{R}$, the following holds:
\begin{equation}\label{main equation 0}
\begin{aligned}
\frac{d}{dt} \int_{\Omega} \frac{|u - \tilde{u}^{-X}|^2}{2} \, d\xi \, dy
&= \dot{X} \int_{\Omega} (u^{X} - \tilde{u}) \, \tilde{u}' \, d\xi \, dy  - \frac{1}{2} \int_{\Omega} |u^{X} - \tilde{u}|^2 \, \tilde{u}' \, d\xi \, dy - \varepsilon_1 \int_{\Omega} |(u^{X} - \tilde{u})_\xi|^2 \, d\xi \, dy \\
&\qquad - \varepsilon_2 \int_{\Omega} |(u^{X} - \tilde{u})_y|^2 \, d\xi \, dy ,
\end{aligned}
\end{equation}
    where \(u^{X}(t,\xi ,y) = u(t,\xi +X(t),y) \) and \(\tilde{u}^{- X} = \tilde{u} (\xi -X(t))\).
\end{lemma}
\begin{proof}

We use \eqref{change of varibale equation} and \eqref{shock equation} to obtain 
\begin{equation}\nonumber
\begin{aligned}
(u - \tilde{u}^{-X})_t
&= \dot{X}(t)\,\tilde{u}_{\xi}^{-X}
+ \sigma\,(u - \tilde{u}^{-X})_{\xi}  - \left( \frac{u^2}{2} - \frac{(\tilde{u}^{-X})^2}{2} \right)_{\xi} + \varepsilon_1 (u - \tilde{u}^{-X})_{\xi\xi}
+ \varepsilon_2 (u - \tilde{u}^{-X})_{yy} \\
&\qquad - \delta (u - \tilde{u}^{-X})_{\xi\xi\xi}
+ v.
\end{aligned}
\end{equation}
Then we compute
\begin{align}\label{3.4}
&\frac{d}{dt} \int_{\Omega} \frac{|u - \tilde{u}^{-X}|^2}{2} \, d\xi \, dy \nonumber \\
=& \int_{\Omega} (u - \tilde{u}^{-X}) (u - \tilde{u}^{-X})_t \, d\xi \, dy,  \nonumber \\
=& \int_{\Omega} (u - \tilde{u}^{-X}) \Big[
\dot{X}\,\tilde{u}^{-X}_{\xi}
+ \sigma (u - \tilde{u}^{-X})_{\xi}
- \Big( \frac{u^2}{2} - \frac{(\tilde{u}^{-X})^2}{2} \Big)_{\xi} 
+ \varepsilon_1 (u - \tilde{u}^{-X})_{\xi\xi}
 \nonumber \\
&{\color{white}...............................................................}+ \varepsilon_2 (u - \tilde{u}^{-X})_{yy}- \delta (u - \tilde{u}^{-X})_{\xi\xi\xi}
+ v \Big] \, d\xi \, dy \nonumber \\ \nonumber\\
=& \ \dot{X} \int_{\Omega} (u - \tilde{u}^{-X}) \tilde{u}^{-X}_{\xi} \, d\xi \, dy  - \int_{\Omega} (u - \tilde{u}^{-X})
\left( \frac{u^2}{2} - \frac{(\tilde{u}^{-X})^2}{2} \right)_{\xi} \, d\xi \, dy \nonumber \\
& {\color{white}} +\varepsilon_1 \int_{\Omega} (u - \tilde{u}^{-X}) (u - \tilde{u}^{-X})_{\xi\xi} \, d\xi \, dy  + \varepsilon_2 \int_{\Omega} (u - \tilde{u}^{-X}) (u - \tilde{u}^{-X})_{yy} \, d\xi \, dy \nonumber\\
&{\color{white}}- \delta \int_{\Omega} (u - \tilde{u}^{-X}) (u - \tilde{u}^{-X})_{\xi\xi\xi} \, d\xi \, dy  + \int_{\Omega} (u - \tilde{u}^{-X}) v \, d\xi \, dy  + \sigma \int_{\Omega} (u - \tilde{u}^{-X})_{\xi} (u - \tilde{u}^{-X}) \, d\xi \, dy.
\end{align}
We now analyze each term in \eqref{3.4} separately as follows:
\begin{equation*}
\begin{aligned}
-\int_{\Omega} (u - \tilde{u} ^{- X}) \left( \frac{u^2}{2} - \frac{(\tilde{u} ^{- X})^2}{2} \right)_\xi\, d\xi\, dy 
&= \int_{\Omega} (u - \tilde{u} ^{- X})_\xi \left( \frac{u^2}{2} - \frac{(\tilde{u}^{ - X})^2}{2} \right)\, d\xi\, dy \\
&= \frac{1}{2} \int_{\Omega} (u - \tilde{u}^{ - X})_\xi \left[ (u - \tilde{u} ^{- X})^2 + 2(\tilde{u}^ {- X})(u - \tilde{u} ^{- X}) \right]\, d\xi\, dy \\
&= \frac{1}{2} \int_{\Omega} \left( (u - \tilde{u} ^{- X})^2 \right)_\xi (\tilde{u} ^{- X})\, d\xi\, dy \\
&= -\frac{1}{2} \int_{\Omega} (u - \tilde{u} ^{- X})^2 (\tilde{u}')^{-X}\, d\xi\, dy.
\end{aligned}
\end{equation*}

Using integration by parts and boundary conditions, we can prove that the dispersion term vanishes since 
\[\delta \int_{\Omega} (u - \tilde{u}^{-X})(u - \tilde{u}^{-X})_{\xi\xi\xi}\, d\xi\, dy =0.\]
Similarly, the dissipation term becomes
\[
\begin{aligned}
&\varepsilon_1 \int_{\Omega} (u - \tilde{u}^{-X}) (u - \tilde{u}^{-X})_{\xi\xi} \, d\xi \, dy
+ \varepsilon_2 \int_{\Omega} (u - \tilde{u}^{-X}) (u - \tilde{u}^{-X})_{yy} \, d\xi \, dy \\
&= - \varepsilon_1 \int_{\Omega} |(u - \tilde{u}^{-X})_{\xi}|^2 \, d\xi \, dy
- \varepsilon_2 \int_{\Omega} |(u - \tilde{u}^{-X})_{y}|^2 \, d\xi \, dy.
\end{aligned}
\]
The last term becomes
\[\sigma \int_{\Omega} (u -\tilde{u}^{-X})_{\xi} \,(u -\tilde{u}^{-X}) \,d\xi\, dy = \frac{\sigma}{2} \int_{\Omega} \left(|u -\tilde{u}^{-X}|^{2}\right)_{\xi} \, d\xi\, dy =0. 
\]
Then, we use a change of variable \(\xi \mapsto \xi +X(t)\) to have
\begin{align} \label{last term main}
\frac{d}{dt} \int_{\Omega} \frac{|u - \tilde{u}^{-X}|^2}{2} \, d\xi \, dy
&= \dot{X} \int_{\Omega} (u^{X} - \tilde{u}) \, \tilde{u}' \, d\xi \, dy  - \frac{1}{2} \int_{\Omega} |u^{X} - \tilde{u}|^2 \, \tilde{u}' \, d\xi \, dy - \varepsilon_1 \int_{\Omega} |(u^{X} - \tilde{u})_{\xi}|^2 \, d\xi \, dy \nonumber \\
&\quad {\color{white}......}- \varepsilon_2 \int_{\Omega} |(u^{X} - \tilde{u})_{y}|^2 \, d\xi \, dy+ \int_{\Omega} (u^{X} - \tilde{u}) \, v^{X} \, d\xi \, dy .
\end{align}
where $v^{X}(t,\xi,y)=v(t,\xi+X(t),y)$.

Now we compute the last term in (\ref{last term main}).
\begin{eqnarray}\nonumber
   && \int_{\Omega} (u^{X} -\tilde{u})\, v^{X}\,d\xi\, dy \\\nonumber 
    &=&\int_{\Omega} (u^{X} -\tilde{u}) \left(-\lambda \int_{-\infty}^\xi u^X(t,x_1,y)\,dx_1\right)_{yy}\, d\xi \,dy 
    \\ \nonumber 
 &=& -\lambda \int_{\Omega} (u^{X} -\tilde{u}) \left( \int_{-\infty}^\xi (u^X(t,x_1,y)-\tilde{u}(x_1))\,dx_1\right)_{yy}\, d\xi \,dy \nonumber \\
 &=& -\lambda \int_{\mathbb{T}^1} \int_{\mathbb{R}} \left( \int_{-\infty}^\xi (u^{X}(t,x_1,y) -\tilde{u}(x_1)) \,dx_1\right)_\xi \left( \int_{-\infty}^\xi (u^X(t,x_1,y)-\tilde{u}(x_1)) \,dx_1\right)_{yy} \,d\xi \,dy \nonumber \\
 &=& \lambda \int_{\mathbb{T}^1} \int_{\mathbb{R}} \left( \int_{-\infty}^\xi (u^{X}(t,x_1,y) -\tilde{u}(x_1))\, dx_1\right)_{\xi y} \left( \int_{-\infty}^\xi (u^X(t,x_1,y)-\tilde{u}(x_1))\, dx_1\right)_{y} \,d\xi\, dy \nonumber \\ 
 && -\lambda  \int_{\mathbb{R}} \left( \int_{-\infty}^\xi (u^{X}(t,x_1,1) -\tilde{u}(x_1)) \,dx_1\right)_\xi \left( \int_{-\infty}^\xi (u^X(t,x_1,1)-\tilde{u}(x_1))\, dx_1\right)_{y} \,d\xi \nonumber \\
 && +\lambda  \int_{\mathbb{R}} \left( \int_{-\infty}^\xi (u^{X}(t,x_1,0) -\tilde{u}(x_1))\, dx_1\right)_\xi \left( \int_{-\infty}^\xi (u^X(t,x_1,0)-\tilde{u}(x_1)) \,dx_1\right)_{y} \,d\xi .\nonumber 
\end{eqnarray} 
Using integration by parts and the boundary conditions, we get 
\begin{eqnarray} \label{Extra term bound}
    \int_{\Omega} (u^{X} -\tilde{u})\, v^{X}\,d\xi\, dy 
    = \frac{\lambda}{2} \int_{\mathbb{T}^{1}} \left(\left( \int_{\mathbb{R}} (u^X-\tilde{u})\, d\xi \right)_y \right)^2 dy 
    = 0,
\end{eqnarray}
where the last step follows from \eqref{zeromean2}. This completes the proof of Lemma $\ \ref{d_dtlemma}$.
\end{proof}

Now we choose $\dot X= -\frac{M}{2s} \left(\int_{\Omega}(u^{X} - \tilde{u})\,\tilde{u}'\,d\xi\, dy \right)$ as in \eqref{lipstich} or \eqref{shift non monotone} with different $M$ for monotone and non-monotone cases, then we have 
\begin{align} \label{main equation}
  & \frac{d}{dt} \int_{\Omega} \frac{|u - \tilde{u}^{-X}|^2}{2} \,d\xi\, dy =- \frac{M}{2s} \left(\int_{\Omega}(u^{X} - \tilde{u})\,\tilde{u}'\,d\xi\, dy \right)^2 - \frac{1}{2}\int_{\Omega}|u^{X}-\tilde{u}|^2\,\tilde{u}' \,d\xi\, dy \nonumber \\
  &{\color{white}........................................}- \varepsilon_1 \int_{\Omega}|(u^{X} -\tilde{u})_{\xi}|^{2}  \,d\xi\, dy 
     - \varepsilon_2 \int_{\Omega}|(u^{X} -\tilde{u})_{y}|^{2}  \,d\xi\, dy.
\end{align} 

\section{$L^2$-contraction for the monotone dispersive shock}
\setcounter{equation}{0}

Throughout this section we assume $ 0<\frac{\d(u_- -\ u_+)}{2\e_1^2} \leq \frac{1}{4}$, so the traveling shock wave is monotonically decreasing from $u_-$ to $ u_+$.



\vspace{0.2cm}
\subsection{Key inequalities}
\begin{lemma}
Assume that 
\begin{equation}
    \gamma > \frac{1}{2},\quad \varepsilon_1^2 >  \frac{4\gamma^2\delta s}{2\gamma-1},\quad \lambda_1 < \sqrt{2} - 1. \nonumber
\end{equation}
Let $\tilde{u}$ be the viscous-dispersive shock satisfying the equation \eqref{shock equation}.
Then, the following inequality holds:
\begin{equation} \label{main lemma}
  \lambda_1\ (s - \tilde{u}(\xi))\ (\tilde{u}(\xi) + s) \le -\varepsilon_1\ \tilde{u}'(\xi) \le \gamma\ (s - \tilde{u}(\xi))\ (\tilde{u}(\xi) + s), \qquad \forall \xi \in \mathbb{R}. 
\end{equation}
\end{lemma}
\begin{proof}
We follow the proof of Lemma 2.3 in \cite{chen2026mono} to obtain the lower bound of $ -\varepsilon_1 \tilde{u}'(\xi) $. 

    For the upper bound, we consider the parameterization in terms of $a=\tilde{u}$ as follows:
    \begin{equation}
        h(a) = \tilde{u}_\xi, \ \ p(a) = \frac{\gamma}{\varepsilon_1}\ (a -s)\ (a + s). \nonumber
    \end{equation}
The function h(a) is well-defined since $\tilde{u}$ is monotonically decreasing. We need to show 
\begin{equation}\label{h(a) > p(a)}
    h(a) \ge p(a),  \qquad \forall a \in [-s,s].
\end{equation}
Since \begin{equation} \label{h(s)=0}
    h(-s) = p(-s) = h(s) = p(s) = 0,
\end{equation}  
the inequality is satisfied at the end points. 

We compute the derivative of these functions at the endpoints. Note that since we assume $\sigma=0$, to find the derivative of the function $h$, we integrate (\ref{shock equation}) over $(-\infty, x)$ to find that
\begin{equation}\label{h'(a)}
 \quad h'(a) = \frac{d\tilde{u}_\xi}{d\tilde{u}} = \frac{\tilde{u}_{\xi\xi}}{\tilde{u}_\xi} = \frac{1}{\delta} \left[ \varepsilon_1 - \frac{1}{2h(a)} (a - s)(a + s) \right].
\end{equation}
Using L'H\^{o}pital's rule, we get 
\begin{align}
\lim_{a \to -s} h'(a) &= \lim_{a \to -s} \frac{1}{\delta} \left[ \varepsilon_1 - \frac{(a - s)}{2} \frac{(a + s)}{h(a)} \right] 
\quad \Rightarrow \quad h'(-s) = \frac{1}{\delta} \left[ \varepsilon_1 + \frac{s}{ h'(-s)} \right], \nonumber\\
\lim_{a \to s} h'(a) &= \lim_{a \to s} \frac{1}{\delta} \left[ \varepsilon_1 - \frac{(a + s)}{2} \frac{(a - s)}{h(a)} \right] 
\quad \Rightarrow \quad h'(s) = \frac{1}{\delta} \left[ \varepsilon_1 - \frac{s}{ h'(s)} \right]. \nonumber
\end{align}
It follows that
\begin{equation}
\delta (h'(-s))^2 - \varepsilon_1 (h'(-s)) - s = 0, \qquad \delta (h'(s))^2 - \varepsilon_1 (h'(s)) + s = 0. \nonumber
\end{equation}
Then, since $h'(-s) < 0$ and (\ref{h'(a)}) indicates that $h'(a) < \frac{\varepsilon_1}{\delta}$ on $(-s, s)$, we have
\begin{equation}
h'(-s) = \frac{\varepsilon_1 - \sqrt{\varepsilon_1^2 + 4\delta s}}{2\delta}, \qquad h'(s) = \frac{\varepsilon_1 - \sqrt{\varepsilon_1^2 - 4\delta s}}{2\delta}.\nonumber
\end{equation}
Since $\gamma>\frac{1}{2}$, we have 
\begin{eqnarray} \label{less than}
4\delta s \varepsilon_1^2 (1-2\gamma) &<& 16 \delta^2 \gamma^2 s^2 ,\nonumber \\
\varepsilon_1^4 + 4\delta s \varepsilon_1^2 &<& \varepsilon_1^4 + 16\delta^2 \gamma^2 s^2 + 8 \varepsilon_1^2 \delta \gamma s, \nonumber \\
\varepsilon_1 \sqrt{\varepsilon_1^2 + 4\delta s} &<& \varepsilon_1^2 + 4\delta \gamma s, \nonumber \\
-h'(-s) = \frac{\sqrt{\varepsilon_1^2 + 4\delta s} - \varepsilon_1}{2\delta} &<& \frac{2\gamma s}{\varepsilon_1} = -p'(-s) .
\end{eqnarray}
Since $\varepsilon_1^2 >  \frac{4\gamma^2\delta s}{2\gamma-1}$, 
\begin{align}\label{greater than}
4\varepsilon_1^2 s\delta(2\gamma - 1) &> 16\gamma^2 s^2 \delta^2 ,\nonumber\\
\varepsilon_1^4 - 4\varepsilon_1^2 s \delta &> \varepsilon_1^4 + 16\gamma^2 s^2 \delta^2- 8\gamma s \delta \varepsilon_1^2, \nonumber\\
\varepsilon_1 \sqrt{\varepsilon_1^2-4s\delta} &> \varepsilon_1^2-4\gamma s\delta ,\nonumber \\
4\gamma s \delta &> \varepsilon_1^2 - \varepsilon_1 \sqrt{\varepsilon_1^2 - 4s \delta} ,\nonumber\\
p'(s)= \frac{2\gamma s}{\varepsilon_1} &> \frac{\varepsilon_1 - \sqrt{\varepsilon_1^2 - 4s\delta}}{2\delta} = h'(s). 
\end{align}
Now we assume that (\ref{h(a) > p(a)}) is not true. Then, by (\ref{h(s)=0}), (\ref{less than}), (\ref{greater than}),  there exist two points $b$ and $c$ with $-s < b < c < s$ such that
\begin{align*}
h(b) &= p(b), \quad h'(b) - p'(b) \leq 0, \\
h(c) &= p(c), \quad h'(c) - p'(c) \geq 0.
\end{align*}
We then consider a function $g : [-s, s] \to \mathbb{R}$ which is defined as
\begin{equation}
g(a) := \frac{1}{\delta} \left[ \varepsilon_1 - \frac{\varepsilon_1}{2\gamma} \right] - \frac{2\gamma a}{\varepsilon_1}.
\end{equation}
We have
\begin{align*}
g(b)   &= \frac{1}{\delta} \left[ \varepsilon_1 - \frac{\varepsilon_1}{2\gamma} \right] - 2\gamma \frac{b}{\varepsilon_1}
       = \frac{1}{\delta} \left[ \varepsilon_1 - \frac{1}{2p(b)} (b - s)(b + s) \right] - 2\gamma \frac{b}{\varepsilon_1} = h'(b) - p'(b) \le 0,  \\
g(c)   &= \frac{1}{\delta} \left[ \varepsilon_1 - \frac{\varepsilon_1}{2\gamma} \right] - 2\gamma \frac{c}{\varepsilon_1}
       = \frac{1}{\delta} \left[ \varepsilon_1 - \frac{1}{2p(c)} (c - s)(c + s) \right] - 2\gamma \frac{c}{\varepsilon_1} = h'(c) - p'(c) \ge 0. 
\end{align*}
Thus, $g(b) \le 0 \le g(c)$. This is a contradiction to the fact that g is a strictly decreasing function. This completes the proof of \eqref{h(a) > p(a)} and \eqref{main lemma}.
\end{proof}

\begin{lemma}
    The first derivative of the monotone 
    shock $\tilde{u}$ satisfies the following estimates:
    
    \begin{equation} \label{total variation for monotone}
    \int_{\mathbb{R}} |\tilde{u}'(\xi)| \, d\xi = 2s \quad \text{and}  \quad \sup_{\xi \in \mathbb{R}} |\tilde{u}'(\xi)| \le \frac{s^2}{2\varepsilon_1}.
    \end{equation}
    
\end{lemma}
\begin{proof}
    Since $\tilde{u}$ is monotone, the first equation is trivial. By integrating (\ref{shock equation}) from $-\infty$ to $\x$, we get 
\[\frac{1}{2} (\tilde{u}^2(\x) -s^2) = \varepsilon_1 \tilde{u}'(\x) - \delta \tilde{u}''(\x).\]
Since $\tilde{u}'$ attains maximum when $\tilde{u}''=0$, we get  
\begin{equation}
    \sup_{\xi \in \mathbb{R}} |\tilde{u}'(\xi)| \le \sup_{\xi \in \mathbb{R}} \frac{|\tilde{u}^2(\xi)-s^2|}{2\varepsilon_1} \le \frac{s^2}{2\varepsilon_1}.
\end{equation}

\end{proof}

\subsection{Proof of Theorem \ref{Main theorem for monotone}}
Note that when $u_-=-u_+=s>0$, $\sigma=0$, so $\xi=x$.
We can rewrite (\ref{main equation}) with respect to the following variables:
\[w=(u^{X} - \tilde{u})\circ(\x,y)  , \quad z = (u_{-} - \tilde{u}),\quad W = (u^{X} - \tilde{u})\circ (z,y).\]
Since \(\tilde{u}'< 0\), the change of variable \(\x\mapsto z\in [0,2s]\) satisfying
\begin{equation}
\label{dz}
  \frac{dz}{d\xi} = -\tilde{u}' \quad  
\end{equation}
is well-defined for $\xi\in (-\infty,\infty)$.

Since it holds from (\ref{lipstich}) that
\[\dot{X} = -\frac{1}{4s}\int_{\mathbb{T}^{1}}\int_{0}^{2s}(u(t,\xi +X(t),y) - \tilde{u}(\xi))\, \tilde{u}'(\xi)\,d\xi\, dy.\]
From (\ref{dz}), we get
\begin{equation} \label{X in terms of W bar}
    \dot{X} = \frac{1}{4s}\int_{\mathbb{T}^{1}}\overline{W} \,dy,\qquad \int_{\Omega}(u^{X} - \tilde{u})\,\tilde{u}' \,d\xi\, dy = - \int_{\mathbb{T}^{1}}\overline{W} \,dy, \quad 
\end{equation}
where \(\overline{W}= \int_{0}^{2s}W \,dz\). 

For the second term in (\ref{main equation}),
\beq \label{B bound}
-\frac{1}{2}\int_{\Omega}|u^{X}-\tilde{u}|^2\, \tilde{u}'\,d\xi \, dy 
= \frac{1}{2}\int_{\mathbb{T}^{1}}\int_{0}^{2s}W^{2}\,dz\,dy.
\eeq
From 
(\ref{main lemma}), we get
\begin{equation} \label{lower bound of dz}
 \frac{\lambda_1 z(2s-z)}{\varepsilon_1} \le \frac{dz}{d\xi} \le \frac{\gamma z(2s-z)}{\varepsilon_1}.
\end{equation}
For the third term in (\ref{main equation}), by using (\ref{lower bound of dz}), we get
\begin{align} \label{G bound}
 &\varepsilon_1\int_{\mathbb{T}^{1}}\int_{\mathbb{R}}|(u^{X} - \tilde{u})_{\xi}|^{2}\,d\xi\, dy +\varepsilon_2\int_{\mathbb{T}^{1}}\int_{\mathbb{R}}|(u^{X} - \tilde{u})_{y}|^{2}\,d\xi\, dy  \nonumber \\ &=\varepsilon_1\int_{\mathbb{T}^{1}}\int_{0}^{2s}|\partial_{z}W|^{2}\frac{dz}{d\xi} \,dz\, dy  + \varepsilon_2\int_{\mathbb{T}^{1}}\int_{\mathbb{R}}|\partial_{y}w|^{2}\, d\x \,dy \nonumber\\
&\geq \lambda_1 \int_{\mathbb{T}^{1}}\int_{0}^{2s}z(2s - z)|W_{z}|^{2} \,dz\, dy   + \varepsilon_2 \int_{\mathbb{T}^{1}}\int_{\mathbb{R}}|\partial_{y}w|^{2} \,d\x\, dy.
\end{align} 
Using (\ref{X in terms of W bar}), (\ref{B bound}), (\ref{G bound}), we get
\begin{align}
&
   - \frac{1}{4s} \left(\int_{\Omega}(u^{X} - \tilde{u})\tilde{u}' \,d\xi\, dy \right)^2 - \frac{1}{2}\int_{\Omega}|u^{X}-\tilde{u}|^2\tilde{u}' \,d\xi\,dy- \varepsilon_1 \int_{\Omega}|(u^{X} - \tilde{u})_{\xi}|^{2} \, d\xi \,dy \nonumber \\
   &- \varepsilon_2 \int_{\Omega}|(u^{X} - \tilde{u})_{y}|^{2} \, d\xi \,dy  \nonumber\\
&\leq -\frac{1}{4s}\left(\int_{\mathbb{T}^{1}}\overline{{W}} \,dy\right)^{2}+\frac{1}{2}\int_{\mathbb{T}^{1}}\int_{0}^{2s}W^{2}\,dz\, dy -\lambda_1 \int_{\mathbb{T}^{1}}\int_{0}^{2s}z(2s - z)|W_{z}|^{2} \,dz\, dy  \nonumber \\
& {\color{white}....}-  \varepsilon_2 \int_{\mathbb{T}^{1}}\int_{\mathbb{R}}|\partial_{y}w|^{2} \,d\x\, dy
  \nonumber \\
&= -\frac{1}{4s}\left(\int_{\mathbb{T}^{1}}\overline{{W}}\,dy\right)^{2}+\frac{1}{4s}\int_{\mathbb{T}^{1}}\overline{{W}}^{2}dy +\frac{1}{2}\int_{\mathbb{T}^{1}}\left(\int_{0}^{2s}W^{2}dz-\frac{1}{2s}\overline{{W}}^{2}-\int_{0}^{2s}\frac{1}{2}z(2s-z)|W_{z}|^{2}dz\right)dy  \nonumber\\
& {\color{white}....}+\left(\frac{1}{4}-\lambda_1\right)\int_{\mathbb{T}^{1}} \int_{0}^{2s}z(2s-z)|W_{z}|^{2}\,dz\,\, dy\,  -\varepsilon_2 \int_{\mathbb{T}^{1}}\int_{\mathbb{R}}|\partial_{y}w|^{2} d\x dy.
 \quad  \nonumber 
\end{align}
Note that from (\ref{poincare0}), we have

\[ \int_{0}^{2s}W^2 \,dz- \frac{1}{2s}\overline{{W}}^{2} - \frac{1}{2}\int_{0}^{2s}z(2s - z)|W_{z}|^{2}\,dz\leq 0.\]
Using equation \eqref{G bound}, together with the Poincaré inequality on $\mathbb{T}^1$, we obtain a bound for the right hand side of \eqref{main equation} as follows:

  \begin{align}
  &-\frac{1}{4s} \left(\int_{\Omega}(u^{X} - \tilde{u})\tilde{u}'\,d\xi\, dy \right)^2 
- \frac{1}{2}\int_{\Omega}|u^{X}-\tilde{u}|^2\tilde{u}' \,d\xi\,dy
- \varepsilon_1 \int_{\Omega}|(u^{X} - \tilde{u})_{\xi}|^{2} \, d\xi\, dy \nonumber \\
&\qquad- \varepsilon_2 \int_{\Omega}| (u^{X} - \tilde{u})_{y}|^{2} \, d\xi\, dy \nonumber \\
 &\leq -\frac{1}{4s}\left(\int_{\mathbb{T}^{1}}\overline{W}\,dy\right)^{2}
+\frac{1}{4s}\int_{\mathbb{T}^{1}}\overline{W}^{2}\,dy +\left(\frac{1}{4}-\lambda_1\right)\int_{\mathbb{T}^{1}} \int_{0}^{2s}z(2s-z)|W_{z}|^{2} \,dz \,dy \nonumber \\
&\qquad
-\varepsilon_2 \int_{\mathbb{T}^{1}}\int_{\mathbb{R}}|\partial_{y}w|^{2}\, d\xi\, dy \nonumber \\
&\leq \frac{1}{4s}\int_{\mathbb{T}^{1}} \left(\overline{W} - \int_{\mathbb{T}^{1}} \overline{W}\right)^2 \,dy -\varepsilon_2 \int_{\mathbb{T}^{1}}\int_{\mathbb{R}}|\partial_{y}w|^{2} \,d\x\, dy +\left(\frac{\varepsilon_1}{4\lambda_1}-\varepsilon_1\right)\int_{\mathbb{T}^{1}} \int_{\mathbb{R}}|w_{\xi}|^{2}\,d\xi\, dy
 \quad  \nonumber \\
&\le\frac{1}{16s\pi^2}  \int_{\mathbb{T}^1} |\overline{W}_{y}|^2 dy  -\varepsilon_2 \int_{\mathbb{T}^{1}}\int_{\mathbb{R}}|\partial_{y}w|^{2}\, d\x\, dy +\left(\frac{\varepsilon_1}{4\lambda_1}-\varepsilon_1\right)\int_{\mathbb{T}^{1}} \int_{\mathbb{R}}|w_{\xi}|^{2}\,d\xi\, dy
 \quad  \nonumber\\
&\le \frac{1}{16s\pi^2} \big( \int_{\mathbb{T}^1} \int_{\mathbb{R}}|\partial_{y}{w}|^2 d\xi\, dy \big ) \big( \int_{\mathbb{R}} |\tilde{u}'|^2  d\xi\big) -\varepsilon_2 \int_{\mathbb{T}^{1}}\int_{\mathbb{R}}|\partial_{y}w|^{2} d\x dy +\frac{\varepsilon_1(1-4\lambda_1)}{4\lambda_1}\int_{\mathbb{T}^{1}} \int_{\mathbb{R}}|w_{\xi}|^{2}d\xi dy
 \quad  \nonumber \\
 &\le\frac{\sup_{\mathbb{R}} |\tilde{u}'|}{16s\pi^2} \big( \int_{\mathbb{T}^1} \int_{\mathbb{R}}|\partial_{y}{w}|^2 d\xi dy \big ) \big( \int_{\mathbb{R}} |\tilde{u}'|  d\xi\big) -\varepsilon_2 \int_{\mathbb{T}^{1}}\int_{\mathbb{R}}|\partial_{y}w|^{2} d\x dy +\frac{\varepsilon_1(1-4\lambda_1)}{4\lambda_1}\int_{\mathbb{T}^{1}} \int_{\mathbb{R}}|w_{\xi}|^{2} d\xi dy
 \nonumber \\
&\le \left(\frac{s^2}{16\pi^2\varepsilon_1}- \varepsilon_2\right) \int_{\mathbb{T}^1} \Big(\int_{\mathbb{R}} |\partial_{y}{w}|^2 \,d\x \,dy \Big) +\frac{\varepsilon_1(1-4\lambda_1)}{4\lambda_1}\int_{\mathbb{T}^{1}} \int_{\mathbb{R}}|w_{\xi}|^{2} \,d\xi\, dy \label{less than zero}
\end{align}
where we used \eqref{total variation for monotone} in the last step. From \eqref{mono_ass}, we have \(C^* = 2\varepsilon_2 - \frac{s^2}{8\pi^2 \varepsilon_1} > 0\). Moreover, since \(\lambda_1 > \frac{1}{4}\), it follows that
\(
C_* = 2 - \frac{1}{2\lambda_1} > 0
\). Using \eqref{main equation} and \eqref{less than zero}, we obtain
\begin{align}
\frac{d}{dt} \int_{\Omega} \frac{|u - \tilde{u}^{-X}|^2}{2}\, d\xi \, dy
&+ \frac{C_{*}\varepsilon_1}{2} \int_{\Omega} \left( \partial_x \big(u(t,x,y)-\tilde{u}(x  - X(t))\big) \right)^2 \, dx\,dy \nonumber\\
&+ \frac{C^{*}}{2} \int_{\Omega} \left( \partial_y \big(u(t,x,y)-\tilde{u}(x - X(t))\big) \right)^2 \, dx\,dy \label{final equation monotone case}
\le 0.
\end{align}
By integrating \eqref{final equation monotone case} over the interval [0,T], we obtain \eqref{theorem1}.

Using \eqref{theorem1}, the proof of \eqref{tstabilitymonotone} is analogous to that of Theorem 2.3 in \cite{Hur}. Consequently, \eqref{tstabilitymonotone} implies that 
\begin{equation}\nonumber\label{Xdotestimate}
|\dot{X}(t)|
\leq \| u(t,\cdot,\cdot) - \tilde{u}(\cdot - X(t)) \|_{L^4(\Omega)}
 \|\tilde{u}' \|_{L^\frac{4}{3}(\mathbb{R})}
\;\longrightarrow\; 0
\quad \text{as } t \to \infty,
\end{equation}
since \( \|\tilde{u}'\|_{L^\frac{4}{3}(\mathbb{R})}\) is bounded.
This completes the proof of Theorem \ref{Main theorem for monotone}.


\section{$L^2$-contraction for the non-monotone dispersive shock}
\setcounter{equation}{0}
Throughout this section we 
assume  \eqref{cond1} with \(u_-=-u_+=s>0\), and  $\varepsilon_1 =1$.
Let \(\util\) be the associated non-monotone viscous-dispersive shock of \eqref{shock equation} which connects \(u_-\) and \(u_+\).  As shown in Section 4 of \cite{chen2026uniform}, when \(A=\frac{1}{2}\), we have the following:
\begin{equation} \label{u0-exp}
u_0 \le  1.0601s, \qquad
\frac{u_{i-1}-s}{s-u_i} \ge \r_* \coloneqq 4.64, \qquad
\frac{s-u_i}{u_{i+1}-s} \ge \r^* \coloneqq 4.77,
\end{equation}
for each odd \(i\in\NN\).

\subsection{Key Inequalities}
Let's first review the properties of the traveling wave solution in \cite{chen2026uniform}.
\begin{proposition} [{\cite[Proposition 4.1]{chen2026uniform}}]\label{prop:key} 
The following holds: for each \(i\in\NN\cup\{0\}\) even,
\begin{equation} \label{key:dec}
-\util'(\x) \ge \lbar_i(u_i-\util(\x))(\util(\x)-u_{i-1}), \qquad \forall \x\in(\x_i,\x_{i-1}),
\end{equation}
and for each \(i\in\NN\) odd, 
\begin{equation} \label{key:inc}
\util'(\x) \ge \lbar_i(u_{i-1}-\util(\x))(\util(\x)-u_i), \qquad \forall \x\in(\x_i,\x_{i-1}),
\end{equation}
where \(\x_{-1}=+\infty\) and \(u_{-1}=\util(\x_{-1})=-s\), and \(\lbar_i\) are given by 
\begin{equation} \label{key:l}
\lbar_0 = 0.355, \qquad
\lbar_1 = 9.60, \qquad
\lbar_{2n} = (1.94) \r_*^{2n}, \qquad
\lbar_{2n+1} = (0.51)^{-1} \r_*^{2n+1}.
\end{equation}
\end{proposition}

We define the interval $J_i:=(\x_i,\x^i)$ with \(\x^i\in (\x_{i-1},\x_{i-2})\) satisfying \(\util(\x^i)=u_i\).
\begin{proposition} [{\cite[Proposition 4.2]{chen2026uniform}}] \label{prop:L2} 
The viscous-dispersive shock profile satisfies the following \(L^2\) estimate:
\begin{equation} \label{L2B}
\frac{1}{\abs{u_i-u_{i-1}}}\int_{J_i} (\util-u_i)^2 d\x
\le
\begin{cases}
0.178, & i=1,\\
0.81(\r_*)^{-i}, & i\ge2 \text{ even},\\
0.82(\r_*)^{-i}, & i\ge3 \text{ odd}.
\end{cases}
\end{equation}
Moreover, let \(\x_s\in(\x_1,\x_0)\) be the unique point satisfying \(\util(\x_s)=s\).
Then the following holds:
\begin{equation} \label{L2B-S}
\int_{-\infty}^{\x_s} (\util-s)^2 d\x \le 0.001s.
\end{equation}
\end{proposition}

Using these estimates, we can obtain some useful estimates on $\tilde{u}'$.

\begin{proposition}
The first derivative of the viscous-dispersive shock profile is bounded and satisfies the following $L^1$ estimate:
\begin{equation} \label{total variation}
    \int_{\mathbb{R}} |\tilde{u}'(\xi)| \, d\xi \le 2.034s,  \quad and  \quad \sup_{\xi \in \mathbb{R}} |\tilde{u}'(\xi)| \le 0.5 s^2.
\end{equation}
\end{proposition}
\begin{proof}
Let $u_{-1}=-s$, then:
 \[
\begin{aligned}
\int_{\mathbb{R}} |\tilde{u}'(\xi)| \, d\xi &= (u_0 + s) + (u_0 - u_1) + (u_2 - u_1) + (u_2 - u_3) + (u_4 - u_3) + \cdots \\
&= \lim_{n \to \infty} \left(\sum_{i=0}^n (u_{2i}- u_{2i-1}) + \sum_{i=0}^n (u_{2i} - u_{2i+1})\right) \\
&\leq  2\lim_{n \to \infty}\left( \sum_{i=0}^{n+1} |u_{2i} - s| + |u_{2i-1} - s|\right) -2s \\
&= 2 \sum_{i=0}^\infty |u_{i} - s| + 2s \\
&\le 2 |u_0-s| \sum_{i=0}^\infty \frac{1}{\rho_*^i} +2s\\
& \le 0.1202 s \frac{ 1}{\rho_*-1} +2s \le 2.034 s ,
\end{aligned}
\]
where we used Theorem \ref{thm:shock} to get $|u_i-s| \le \frac{1}{\rho_*^i} |u_0-s|, \ \  \forall i \in \mathbb{N}$.

By integrating (\ref{shock equation}) from $-\infty$ to $\x$, we get 
\[\frac{1}{2} (\tilde{u}^2(\x) -s^2) = \varepsilon_1 \tilde{u}'(\x) - \delta \tilde{u}''(\x)\].
Since $\tilde{u}'$ attains a maximum when $\tilde{u}''=0$ and $\varepsilon_1=1$, we get 
\begin{equation}\nonumber
    \sup_{\xi \in \mathbb{R}} |\tilde{u}'(\xi)| \le \sup_{\xi \in \mathbb{R}} \frac{|\tilde{u}^2(\xi)-s^2|}{2}  \le \frac{1}{2} s^2.
\end{equation}

\end{proof}

\begin{theorem}\label{induction step thm}
Under the assumption in Theorem \ref{main theorem for non monotone}, we have
\begin{equation} \label{induction step}
       -\frac{M}{2s} \int_{\mathbb{T}^1}\left( \int_{\mathbb{R}} w \tilde{u}' \,d\xi \right)^2 \,dy- \frac{1}{2} \int_{\mathbb{T}^1} \int_{\mathbb{R}} w^2 \tilde{u}' \,d\xi\, dy - \int_{\mathbb{T}^1}\int_{\mathbb{R}} (w_{\xi})^2 \,d\xi \,dy \leq - \frac{1}{10}\int_{\mathbb{T}^1}\int_{\mathbb{R}} (w_{\xi})^2 \,d\xi \,dy,
\end{equation}
where $w=u^X-\tilde{u}$ and $M=\frac{4}{3}$.
\end{theorem}

For the proof of the theorem,
we follow the same induction process as in Section 4 of \cite{chen2026uniform}. You can find it in the Appendix~\ref{app}.

\subsection{Proof of Theorem \ref{main theorem for non monotone}}
From equation (\ref{main equation}), we aim to show 
    \begin{eqnarray} 
        \frac{-M}{2s} \left(\int_{\Omega}(u^{X} - \tilde{u})\tilde{u}'\,d\xi\, dy \right)^2 - \frac{1}{2}\int_{\Omega}|u^{X}-\tilde{u}|^2\tilde{u}' \,d\xi\,dy- \int_{\Omega}|(u^{X} -\tilde{u})_\xi|^{2} \ \,d\xi\, dy \nonumber \\ 
        - \e_2 \int_{\Omega}|(u^{X} -\tilde{u})_y|^{2} \ \,d\xi\, dy   \le 0.\nonumber
    \end{eqnarray}
    for $M=\frac{4}{3}$.
    
Using (\ref{induction step}) and  Poincaré inequality on $\mathbb{T}^1$, we have 
\begin{eqnarray} 
          &&   \frac{-M}{2s} \left(\int_{\Omega}w\tilde{u}'\,d\xi\, dy \right)^2 - \frac{1}{2}\int_{\Omega}w^2\tilde{u}' \,d\xi\,dy- \int_{\Omega}|w_\xi|^{2} \,d\xi \,dy- \e_2 \int_{\Omega}|w_y|^{2} \,d\xi \,dy   \nonumber \\
        &\le& \frac{-M}{2s} \left(\int_{\mathbb{T}^1}\int_{\mathbb{R}}w\tilde{u}'\,d\xi \,dy \right)^2 + \frac{M}{2s} \int_{\mathbb{T}^1}\left( \int_{\mathbb{R}} w \tilde{u}' \,d\xi \right)^2 \,dy -\frac{1}{10}\int_{\mathbb{T}^1}\int_{\mathbb{R}}|w_{\xi}|^{2} \,d\xi \,dy - \e_2\int_{\mathbb{T}^1}\int_{\mathbb{R}}|w_y|^{2} \,d\xi \,dy \nonumber \\ 
        &=& \frac{M}{2s}\int_{\mathbb{T}^1} \left( \int_{\mathbb{R}}w\tilde{u}'\,d\xi- \int_{\mathbb{T}^1} \int_{\mathbb{R}}w\tilde{u}'\,d\xi \,dy\right ) ^2 \,dy-\frac{1}{10}\int_{\mathbb{T}^1}\int_{\mathbb{R}}|w_{\xi}|^{2} \,d\xi \,dy - \e_2\int_{\mathbb{T}^1}\int_{\mathbb{R}}|w_y|^{2} \,d\xi \,dy\nonumber  \\ 
        &\le& \frac{M}{8s\pi^2} \int_{\mathbb{T}^1}\left(\int_{\mathbb{R}}w_y \tilde{u}' \,d\xi \right)^2 \,dy -\e_2\int_{\mathbb{T}^1}\int_{\mathbb{R}}|w_y|^{2}\, d\xi \,dy-\frac{1}{10}\int_{\mathbb{T}^1}\int_{\mathbb{R}}|w_\x|^{2} \,d\xi \,dy \nonumber   \\ 
        &\le& \frac{M}{8s\pi^2} \int_{\mathbb{T}^1} \left(\int_{\mathbb{R}}|w_y|^{2} \,d\xi \right) \left( \int_{\mathbb{R}} |\tilde{u}'|^2 \,d\xi \right ) \,dy - \e_2\int_{\mathbb{T}^1}\int_{\mathbb{R}}|w_y|^{2} \,d\xi \,dy -\frac{1}{10}\int_{\mathbb{T}^1}\int_{\mathbb{R}}|w_{\xi}|^{2} \,d\xi\, dy \nonumber  \\
        &\le& \frac{M}{8s\pi^2} \sup_\mathbb{R} |\tilde{u}'| \left(\int_{\mathbb{R}} |\tilde{u}'| \,d\xi \right) \left( \int_{\mathbb{T}^1}\int_{\mathbb{R}}|w_y|^{2} \,d\xi \,dy \right ) - \e_2\int_{\mathbb{T}^1}\int_{\mathbb{R}}|w_y|^{2} \,d\xi \,dy -\frac{1}{10}\int_{\mathbb{T}^1}\int_{\mathbb{R}}|w_{\xi}|^{2} \,d\xi \,dy \nonumber  \\
        &\le& \left( \frac{2.034 }{16\pi^2}  Ms^2  -\e_2 \right ) \int_{\mathbb{T}^1}\int_{\mathbb{R}}|w_y|^{2} \,d\xi \,dy -\frac{1}{10}\int_{\mathbb{T}^1}\int_{\mathbb{R}}|w_{\xi}|^{2} \,d\xi\, dy \label{less than zero nonmonotone},
    \end{eqnarray}
    where we used \eqref{total variation} in the last step.
    
Using \eqref{s_def}, we have \(C^{**}=2\e_2-\frac{2.034 }{8\pi^2}  Ms^2  >0\). Thus, using \eqref{main equation} and \eqref{less than zero nonmonotone}, we obtain 
\begin{align}
\frac{d}{dt} \int_{\Omega} \frac{|u - \tilde{u}^{-X}|^2}{2}\, d\xi \, dy
&+ \frac{1}{10} \int_{\Omega} \left| \partial_x \big(u(t,x,y)-\tilde{u}(x - X(t))\big) \right|^2 \, dx\,dy \nonumber \\
&+ \frac{C^{**}}{2} \int_{\Omega} \left| \partial_y \big(u(t,x,y)-\tilde{u}(x  - X(t))\big) \right|^2 \, dx\,dy \le 0. \label{last equation non monotone}
\end{align}
By integrating \eqref{last equation non monotone} over the interval [0,T], we obtain \eqref{theorem2}.

Using \eqref{theorem2}, the proof of \eqref{tstabiltynonmonotone} is analogous to that of Theorem 2.3 in \cite{Hur}. Consequently, \eqref{tstabiltynonmonotone} implies that
\begin{equation}\nonumber
\label{Xdot_estimate}
|\dot{X}(t)|
\leq \| u(t,\cdot,\cdot) - \tilde{u}(\cdot - X(t)) \|_{L^4(\Omega)}
 \|\tilde{u}'(\xi) \|_{L^\frac{4}{3}(\mathbb{R})}
\;\longrightarrow\; 0
\quad \text{as } t \to \infty,
\end{equation}
since \( \|\tilde{u}'(\xi) \|_{L^{\frac{4}{3}}(\mathbb{R})}\) is bounded. This completes the proof of Theorem \ref{main theorem for non monotone}.

\section{Global existence of solutions}
 This section is devoted to the global existence of solutions to \eqref{KP main equation} in the class $X_T$. 
We consider solutions that have different asymptotic states as $|x| \to \infty$. 
Throughout this section, the coefficients $\varepsilon_1$, $\varepsilon_2$, and $\delta$ do not play a role in the analysis; hence, for simplicity, in this section, we set $$\varepsilon_1 = \delta = \e_2 =1.$$ The proof for the case with other positive constants $\varepsilon_1$, $\e_2$ and $\delta$, is very similar and is omitted.

Assume the initial data $u^0$ satisfies that there exists a monotonically decreasing function $f(x)$ with the same far-field values as $u^0$ when $x\rightarrow \pm \infty$, such that $u^0 - f \in H^s(\Omega)$ for any $s >1$, \(u^0(x,y)=u^0(x,y+1), \forall x\in \mathbb{R},\ y\in \mathbb{T}^1\), and $\int_{\mathbb{R}} (u^0(x,y)-f(x))\,dx = 0$, $\forall y \in \mathbb{T}^1$.
Our goal is to establish the global existence of solutions to \eqref{KP main equation} such that, for any $T > 0$,
\[
u - f \in \Scal_T,
\]
where
\[
\Scal_T :=
\left\{
g : \mathbb{R}_+ \times \Omega \to \mathbb{R}
\;\middle|\;
\begin{aligned}
&g \in C([0,T];H^1(\Omega)) \cap L^2(0,T;H^2(\Omega)), \\
&\int_{\mathbb{R}} g(t,x,y)\,dx = 0,
\quad \forall\, t \in [0,T], \ \text{a.e. } y \in \mathbb{T}^1
\end{aligned}
\right\}.
\]
Observe that the condition $u - f \in \Scal_T$ is equivalent to $u \in X_T$. 
The global existence result is obtained by combining local well-posedness with a suitable priori estimates via a standard continuation argument. 

We begin by stating the local existence lemma.
\begin{lemma}\label{local existence}
Let $f\in C^\infty(\mathbb{R})$ be a smooth, monotone function connecting the end states from $u_-$ to $u_+$. Then, for any $M_0 > 0$, there exists a time $T > 0$ such that if the initial data $u^0(x,y)$ satisfies the following conditions:

\begin{equation} 
\begin{cases}
    \|u^0 - f\|_{H^s(\Omega)} \le M_0, \\
    u^0(x,y+1) = u^0(x,y),\ \ \ \ \ \ \    \forall\, (x,y) \in \Omega, \\
    \int_{\mathbb{R}} u^0(x,y) \, dx = \int_{\mathbb{R}} f(x) \,dx ,\ \  \forall\, y \in \mathbb{T}^1,
\end{cases}
\end{equation}
then there exists a unique solution $u$ to the equation \eqref{KP main equation} on $[0,T]$ satisfying 
\begin{equation}
\begin{cases}
    u - f \in C([0,T]; H^1(\Omega)) \cap L^2(0,T; H^2(\Omega)) \\ 
    u(x,y) = u(x,y+1), \ \ \ \ \forall\, (x,y) \in \Omega.
\end{cases}
\end{equation}
\end{lemma}
\begin{remark}
    This lemma establishes the local existence of solutions to the KP equation \eqref{KP main equation} and shows that if the initial data $u^0$ is periodic in y with period 1, then the corresponding unique solution u remains periodic in y with period 1.
\end{remark}

\begin{proof}
    The proof follows a classical argument and consists of five steps. \\
\step{1}
We begin by introducing the perturbation variable
    \[ w(t,x,y) = u(t,x,y) - f(x),
    \]
    so that the initial condition can be written as
    \[w(0,x,y) =w^0(x,y) = u^0(x,y) - f(x) \in H^s(\Omega)\ \text{for}\ s>1,\]
and w satisfies 
\begin{equation}\label{equation w}
w_t + w_{xxx} - w_{xx} -w_{yy}+ w w_x + f\, w_x + f_x\, w + \lambda\partial^{-1}_x\, w_{yy} = F,
\end{equation}
where \(F := -f\,f_x - f_{xxx} + f_{xx},\ \text{and}\ \partial^{-1}_x \ \text{is inverse partial derivative with respect to x.}
\)
It is clear that
\(
F \in C_c^\infty(\Omega) \subset L^2(\Omega) .
\) 

We define a new space X,
\[
X := \left\{ w \in L^2(\mathbb{R} \times \mathbb{T}^1) 
:\ \int_{\mathbb{R}} w(x,y)\,dx = 0 
\text{ for a.e. } y \right\}.
\]
Now, we define the linear differential operator A:
\[
A = -\partial_x^3+\partial_x^2+\partial_y^2 - \lambda\partial_{x}^{-1}\partial_{yy}, \quad {
D(A) 
= \left\{ w \in X : Aw \in X \right\}.}
\]
By following arguments similar to those in {\cite[Proposition 2.1]{cavalcanti2014global}}, one can show that $A$ generates a contraction $C_0$-semigroup $\{S(t)\}_{t \ge 0}$ on $X$.

Define the translation operator $\tau$ acting on functions $g$ by
\[
\tau(g)(t,x,y) = g(t,x,y+1).
\]
Since $A$ is a linear differential operator, it commutes with $\tau$, i.e., $A\tau = \tau A$. 
Therefore, the semigroup $\{S(t)\}_{t \ge 0}$ generated by $A$ also commutes with $\tau$, i.e., $S(t)\tau = \tau S(t)$. \\

\step{2}
We consider the following linear problem:
\begin{equation}\label{step 2 assumption}
\begin{cases}
    v_t + v_{xxx} - v_{xx}- v_{yy} +\lambda\partial_{x}^{-1}v_{yy} = \bar{f},\quad \text{for } t>0, \\
    v(0,x,y) = v^0(x,y), \\
    \bar{f}(t,x,y) = \bar{f}(t,x,y+1), \\
    v^0(x,y)=v^0(x,y+1) , \\
    \int_{\mathbb{R}} \bar{f} (t,x,y) dx =0,\forall t \in [0,T]  , \forall y\in \mathbb{T}^1, \\
     \int_{\mathbb{R}}v^0(x,y)\,dx =0, \forall y\in \mathbb{T}^1 .
    
\end{cases}
\end{equation}
In this step, we prove that for any $v^0 \in H^s(\Omega)$, for $s>1$, and 
$\bar{f} \in L^2(0,T; L^2(\Omega))$, the corresponding mild solution satisfies
\begin{equation}\label{linear estimate}
\begin{cases}
    v \in C([0,T]; H^1(\Omega)) \cap L^2(0,T; H^2(\Omega)), \\
    v(t,x,y)=v(t,x,y+1) , \\
    \int_{\mathbb{R}} v(t,x,y)\, dx =0 , \forall t>0\ \text{and} \ y \in \mathbb{T}^1 ,
\end{cases}
\end{equation}
and the estimate
\begin{equation}\label{step 2 estimate}
    \|v\|_{L^\infty(0,T; H^1(\Omega))}^2 
+ \|v\|_{L^2(0,T; H^2(\Omega))}^2 
\le C_T \left( 
\|v^0\|_{H^1(\Omega)}^2 
+ \|\bar{f}\|_{L^2(0,T; L^2(\Omega))}^2 
\right)
\end{equation}
holds.

To this end, we construct the mild solution using the semigroup $\{S(t)\}_{t \ge 0}$ and the Duhamel formula:
\[
v(t) = S(t)v^0 + \int_0^t S(t - s)\bar{f}(s)\, ds.
\]
Using the commutation property $S(t)\tau = \tau S(t)$, together with the fact that both $\bar{f}$ and $v^0$ are periodic in $y$ with period $1$, and that the integral operator commutes with $\tau$, we conclude that the solution $v$ is also periodic in $y$ with period $1$. Similarly, using the last two conditions from the equation \eqref{step 2 assumption} together with the properties of the semigroup, 
we can show that $v \in X, \ \ \forall t \in [0,T]$.

The strong continuity of the semigroup $\{S(t)\}_{t \ge 0}$ ensures that the initial condition is satisfied. Using the boundary condition in $x$ and periodic condition in $y$, we can show the following:
\[
\frac{1}{2}\frac{d}{dt} \|v\|_{L^2(\Omega)}^2 
+ \|v_x\|_{L^2(\Omega)}^2 + \|v_y\|_{L^2(\Omega)}^2
= \langle \bar{f}, v \rangle_{L^2, L^2},
\]
\[\frac{1}{2}\,\frac{d}{dt}\|v_x\|^{2}_{L^{2}(\Omega)} + \|v_{xx}\|^{2}_{L^{2}(\Omega)} + \|v_{xy}\|^{2}_{L^{2}(\Omega)} = \langle \bar{f}, -
v_{xx} \rangle_{L^2,L^2},\]
\[\frac{1}{2}\,\frac{d}{dt}\|v_y\|^{2}_{L^{2}(\Omega)} + \|v_{yy}\|^{2}_{L^{2}(\Omega)} + \|v_{yx}\|^{2}_{L^{2}(\Omega)} = \langle \bar{f}, -v_{yy} \rangle_{L^2,L^2}\ .\]
From these equations, we can get 
\[
\frac{d}{dt}\|v\|^{2}_{H^1(\Omega)} + \|v\|_{H^2(\Omega)}^2 \le C\ \left(\|\bar{f}\|_{L^2(\Omega)}^2 + \|v\|^{2}_{H^1(\Omega)} \right) 
\]
where $C$ is some constant. 

Then we apply Gronwall's inequality to obtain the desired estimate \eqref{step 2 estimate}. Thus, we get \eqref{linear estimate}. \\

\step{3} In this step, we begin by establishing estimates for functions in the class $\Scal_T$, prior to introducing the iteration scheme. These estimates will be used to control the approximate solutions produced by the scheme, enabling us to apply the result of Step 2 and obtain uniform bounds.

Let $v,g \in \Scal_T$. For $s > 1$, we will make use of the following basic Sobolev embeddings in the upcoming arguments:
\begin{equation}\label{bfa}
    \|g\|_{L^4(\Omega)} \le C \|g\|_{H^1(\Omega)}, \quad \|g\|_{L^\infty(\Omega)} \le C\|g\|_{H^s(\Omega)} , \quad 
\|g\|_{L^4(\Omega)} \leq C \|g\|_{L^2(\Omega)}^{1/2} \|\nabla g\|_{L^2(\Omega)}^{1/2} .
\end{equation}
Suppose \(F_1, G \in \Scal_T\), then
\begin{align}\nonumber
\|F_1 G_x\|_{L^2(\Omega)}^2 
&\leq \|F_1\|_{L^4(\Omega)}^2 \|G_x\|_{L^4(\Omega)}^2 \\
&\leq C \|F_1\|_{H^1(\Omega)}^2 \bigl(\|G\|_{H^1(\Omega)} \|G\|_{H^2(\Omega)}\bigr).
\end{align}
Thus,
\begin{align}\nonumber
\int_0^T \|F_1 G_x\|_{L^2(\Omega)}^2 \, dt 
&\leq C\|F_1\|_{L^\infty(0,T;H^1(\Omega))}^2 \|G\|_{L^\infty(0,T;H^1(\Omega))} \int_0^T \|G\|_{H^2(\Omega)} \, dt.
\end{align}
By Cauchy-Schwarz in time,
\begin{align}\nonumber
\int_0^T \|G\|_{H^2} \, dt 
&\leq \sqrt{T} \left( \int_0^T \|G\|_{H^2}^2 \, dt \right)^{1/2}.
\end{align}
Thus, we get 
\begin{equation}\label{famain}
    \|F_1G_x\|_{L^2(0,T;L^2(\Omega))} \le C\ T^\frac{1}{4}\  \|F_1\|_{L^\infty(0,T;H^1(\Omega))}\  \|G\|^{\frac{1}{2}}_{L^\infty(0,T;H^1(\Omega))}\ \|G\|^{\frac{1}{2}}_{L^2(0,T;H^2(\Omega))}.
\end{equation}
Using \eqref{famain}, we get 
\begin{equation}\label{fa1}
    \|vv_x\|_{L^2(0,T;L^2(\Omega))} \le C\  T^\frac{1}{4}\ \|v\|^2_{\Scal^T}.
\end{equation}
We derive a sharper estimate for $w^0(w^0)_x$ term: 
\[
\|w^0w^0_x\|_{L^2(\Omega)} \le \|w^0\|_{L^\infty(\Omega)}  \|w^0_x\|_{L^2(\Omega)} \le  \|w^0\|_{H^s(\Omega)} \|w^0\|_{H^1(\Omega)}.
\]
This implies that
\[
\|w^0w^0_x\|_{L^2(0,T;L^2(\Omega))}\le \sqrt{T} \|w^0\|_{H^s(\Omega)} \|w^0\|_{H^1(\Omega)}.
\]
Using \eqref{bfa}, we can also obtain the following inequalilty:
\begin{equation}\label{fa2}
    \|v f_x\|_{L^2(0,T;L^{2}(\Omega))} + \|f v_x\|_{L^2(0,T;L^{2}(\Omega))} 
\le C \sqrt{T}\, \|v\|_{L^\infty(0,T;H^1(\Omega))}.
\end{equation}
Since \(F\in C_c^\infty(\Omega) \subset L^{2}(\Omega)\) and \(F\) is stationary, we have 
\begin{equation}\label{fa3}
    \norm{F}_{L^2(0,T;L^{2}(\Omega))}
=\sqrt{T} \norm{F}_{L^{2}(\Omega)}.
\end{equation}
\step{4} 
We now introduce the iteration scheme as follows.
Firstly, we set \(w^{(0)}(t,x,y)=w^0(x,y)\), and then, for each \(n\in\NN\cup\{0\}\), we define \(w^{(n+1)}\) to be the solution of
\begin{align}
(w^{(n+1)})_t + (w^{(n+1)})_{xxx} - (w^{(n+1)})_{xx} &- (w^{(n+1)})_{yy} + \lambda \partial^{-1}_x\partial_{yy} (w^{(n+1)}) \\
&= -w^{(n)} (w^{(n)})_x -  (w^{(n)})_x f - f_x w^{(n)} \nonumber 
 + F, \label{wn}\\
w^{(n+1)}(0,x,y) &= w^0(x,y).
\end{align}
We define \(\Phi(w^{(n)})=w^{(n+1)}\). It is clear that the map \(\Phi\) is well-defined on \(\Scal_T\) by \textit{Step 2} and \textit{Step 3}.

Note that, since the initial data $w^0$ is assumed to be periodic in $y$ with period $1$ and satisfies $w^0 \in X$, the right-hand side of \eqref{wn} also belongs to $X$ and is periodic in $y$ with period $1$. Consequently, for every $n \in \mathbb{N}$, the function $w^{(n)} \in X$  and remains periodic in $y$ with period $1$. 
Using \eqref{step 2 estimate} and \eqref{fa1} -- \eqref{fa3}, we find that for each \(n\in\NN\),
\begin{eqnarray} 
\|w^{(n+1)}\|_{\Scal_T}
\le C_T\big( \norm{w^0}_{H^1(\Omega)}
+\sqrt{T} (\sqrt{T}\|w^{(n)}\|^{2}_{\Scal_T}
+\|{w^{(n)}}\|_{L^\infty(0,T;H^1(\Omega))}
+1)\big) ,
\end{eqnarray}
\[
\|w^{(1)}\|_{\Scal_T}
\le C_T\big( \norm{w^0}_{H^1(\Omega)}
+\sqrt{T} (\|w^0\|_{H^s(\Omega)} \|w^0\|_{H^1(\Omega)}
+\|{w^{(0)}}\|_{L^\infty(0,T;H^1(\Omega))}
+1)\big). 
\]

 Since $C_T \sim e^T$ and $e^T \to 1$ as $T \to 0$, there exist $T>0$ and $R(M_0)>0$ such that
\[
\|{w^{(n)}}\|_{{\Scal_T}} \le R \;\implies\; \|w^{(n+1)}\|_{\Scal_T} \le R, \qquad \forall n \in \mathbb{N} \cup \{0\}.
\]
Hence, $\Phi$ maps $B_R \subset \Scal_T$ into $B_R$. \\

\step{5} 
In this step, we prove that the sequence $\{w^{(n)}\}$ is Cauchy in $\Scal_T$. 
It suffices to show that
\[
\|w^{(n+1)} - w^{(n)}\|_{\Scal_T}
\le C_T\  T^{\frac{1}{4}}\,(R+1)\,\|w^{(n)} - w^{(n-1)}\|_{\Scal_T}.
\]
Then, by choosing $T>0$ sufficiently small such that 
\[
C_T\  T^{\frac{1}{4}}\,(R+1) < \tfrac{1}{2},
\]
we obtain a contraction, and hence $\{w^{(n)}\}$ is a Cauchy sequence in $\Scal_T$.

Define  \(z^{(n)}\coloneqq w^{(n+1)}-w^{(n)}\). Then it satisfies \(z^{(n)}(0,x,y)=0\) and 
\begin{align*}
&(z^{(n+1)})_t + (z^{(n+1)})_{xxx} - (z^{(n+1)})_{xx} - (z^{(n+1)})_{yy}+ \lambda \partial^{-1}_x\partial_{yy}z^{(n+1)}\\
&\qquad= -z^{(n)}(w^{(n+1)})_x
-w^{(n)}(z^{(n)})_x
-(z^{(n)})_xf
-f_x z^{(n)} .
\end{align*}
Using \eqref{famain}, we can obtain
\[
\| z^{(n)} (w^{(n+1)})_x \|_{L^2(0,T; L^{2}(\Omega)} \le C\ T^{\frac{1}{4}}\ \| z^{(n)}\|_{L^\infty(0,T; H^{1}(\Omega))}\ \|w^{(n+1)}\|_{\Scal_T} ,
\]
\[
\| (z^{(n)})_x w^{(n+1)}\|_{L^2(0,T; L^{2}(\Omega))} \le C\ T^\frac{1}{4}\  \| w^{(n+1)}\|_{L^\infty(0,T; H^{1}(\Omega))}\ \|z^{(n)}\|_{\Scal_T} .
\]
Using \eqref{fa2}, we can obtain 
\[
\|  (z^{(n)})_x f+ f_x z^{(n)} \|_{L^2(0,T; L^{2}(\Omega))} 
\le C\ \sqrt{T}\,\ \|z^{(n)}\|_{L^\infty(0,T; H^1(\Omega))}.
\]
Thus, 
\[
\| \text{R.H.S}\|_{L^2(0,T;L^2(\Omega))} \le C\ T^\frac{1}{4}\ \| {z^{(n)}}\|_{\Scal_T}.
\]
We then apply \eqref{step 2 estimate} to find that 
\[
\|{z^{(n+1)}}\|_{\Scal_T} \le C_T\ T^{\frac{1}{4}}\ (R+1) \|{z^{(n)}}\|_{\Scal_T}.
\]
Thus, by the contraction mapping theorem, there exists a unique fixed point $w$ of the map $\Phi$, which provides a solution to our problem. Moreover, this unique solution is $1$-periodic in $y$.
\end{proof}

Now, we prove the global existence theorem using the  a priori estimate in the following lemma.
\begin{lemma}\label{global exist}
     Let $u_-$ and $u_+$ be two given constant states, and let $f$ be a smooth monotone function connecting $u_-$ and $u_+$. Let $u^0$ be an initial datum such that $u^0 - f \in H^s(\Omega)$ for $s>1$. Assume that $u$ is a solution of \eqref{KP main equation} on $[0,T_0)$ for some $T_0>0$, and that $u \in X_T$ (equivalently, $u - f \in \Scal_T$) for all $T \in (0,T_0)$. Then there exists a constant $C(T_0)$ such that
\begin{equation}\label{H^1 bound}
\sup_{t \in [0,T_0)} \|u - f\|_{H^1(\Omega)} \le C(T_0).
\end{equation}
\end{lemma}
\begin{proof}
Multiply \eqref{equation w} by $w$ to get 
\begin{eqnarray}
&&\int_{\Omega} w_t\ w \,dx\, dy +\int_{\Omega} w_{xxx}w\, dx\, dy -\int_{\Omega}w_{xx} w\, dx\, dy -\int_{\Omega}w_{yy} w\, dx\, dy \nonumber \\
&& = -\int_{\Omega} w^2 w_x \, dx\, dy -\int_{\Omega} f\, ww_x \, dx\, dy - \int_{\Omega}f_x\, w^2 \,dx \,dy -\lambda\int_{\Omega} \partial^{-1}_x\, w_{yy} w\ \,dx\, dy + \int_{\Omega} F w\, dx\,  dy. \nonumber
\end{eqnarray}
Using integration by parts, L.H.S becomes
\begin{align*}
\int_{\Omega} w_t w \, dx \,dy = \frac{1}{2} &\frac{d}{dt} \|w\|_{L^2(\Omega)}^2, \quad
\int_{\Omega} w_{xxx} w \, dx\,dy = -\int_{\Omega} w_{xx} w_x \, dx\,dy = 0, \\
&- \int_{\Omega} (\Delta w) w \, dx\,dy = \int_{\Omega} |\nabla w|^2\, dx\, dy.
\end{align*}
We analyze R.H.S on a term-by-term basis. The first term becomes
\[
\int_{\Omega} w w_x w \,dx\,dy = \frac{1}{3} \int \partial_x(w^3) \, dx\,dy = 0 .
\]
The second term is equal to 
\[
- \int_{\Omega}f\, w w_x\, dx\, dy = - \frac{1}{2}\int_{\Omega}f\,(w^2)_x\, dx\, dy= \frac{1}{2} \int_{\Omega}f_x\, w^2 \, dx\, dy .
\]
Adding this to the third term we get
\[
\frac{1}{2} \int_{\Omega}f_x\, w^2 \, dx\,dy - \int_{\Omega}f_x\, w^2 \, dx\, dy = -\frac{1}{2} \int_{\Omega}f_x\, w^2 \, dx\, dy \le C \|w\|^2_{L^2(\Omega)}.
\]
Using the fact that the operator \(\partial^{-1}_x\partial_{yy}\) is skew symmetric, the second last term becomes
\[
\int_{\Omega} \partial^{-1}_x\, w_{yy} w\, dx\, dy =0 .
\]
Similarly the last term can be bounded by \[\int_{\Omega} F\ w\, dx\,  dy \le C\|w\|^2_{L^2(\Omega)}.\]
Combining all these we get:
\[
\frac{1}{2} \frac{d}{dt} \|w\|_{L^2(\Omega)}^2 +\|\nabla w\|_{L^2(\Omega)}^2\ \le C \|w\|^2_{L^2(\Omega)}.
\]
By Gronwall's inequality, for any interval $[0,T_0]$, we obtain:
\begin{equation}
\sup_{t \in [0,T_0]}\|w(t)\|_{L^2(\Omega)} \le C(T_0,\|w^0\|_{L^2(\Omega)}). \label{eq:L2bound}
\end{equation}
Integrating in time provides the crucial integrability of the gradient:
\begin{equation}
\int_0^T \|\nabla w(t)\|_{L^2(\Omega)}^2 dt \le C(T,\|w^0\|_{L^2(\Omega)}). \label{eq:grad_int}
\end{equation}
Next, we multiply \eqref{equation w} by $-\Delta w$ and integrate, then we get the following:
\begin{eqnarray}
&&\int_{\Omega} w_t\ \Delta w\, dx\, dy +\int_{\Omega} w_{xxx}\Delta w\, dx\, dy -\int_{\Omega}w_{xx}\Delta w\, dx\, dy -\int_{\Omega}w_{yy}\Delta w\, dx\, dy \nonumber \\
&& = -\int_{\Omega} w^2 w_x\ \Delta w\, dx\, dy -\int_{\Omega} f\, ww_x\ \Delta w\, dx\, dy - \int_{\Omega}f_x\, w \Delta w\, dx\, dy \\ && -\lambda\int_{\Omega} \partial^{-1}_x\, w_{yy}\ \Delta w\, dx\, dy + \int_{\Omega} F \Delta w\, dx\, dy. \nonumber
\end{eqnarray}
The LHS becomes 
\begin{align*}
&\int_{\Omega} w_t (-\Delta w) \, dx\,dy = \frac{1}{2} \frac{d}{dt} \|\nabla w\|_{L^2(\Omega)}^2, \quad 
 \int_{\Omega} \Delta w (-\Delta w) \, dx\,dy = \|\Delta w\|_{L^2(\Omega)}^2, \\
&\int_{\Omega} w_{xxx} (-\Delta w) \, dx\,dy = \int_{\Omega} w_{xxx} (-w_{xx} - w_{yy}) \, dx\,dy = -\frac{1}{2} \int_{\Omega} \partial_x (w_{xy}^2) \, dx\,dy = 0.
\end{align*}
We analyze RHS on a term-by-term basis. The first term becomes,
\begin{eqnarray}\nonumber
    -\int_{\Omega} w w_x (w_{xx} + w_{yy}) \,dx\, dy 
    &=& \frac{1}{2}\int_{\Omega} w_x^3 \,dx\, dy + \frac{1}{2}\int_{\Omega} w_x w_y^2 \,dx\, dy \nonumber\\ 
    &=& \frac{1}{2}\int_{\Omega} w_x |\nabla w|^2 \,dx\, dy \nonumber \\
    &\le& \frac{1}{2} \int_{\Omega} |\nabla w|^3 \,dx\, dy. \nonumber
\end{eqnarray}
By the 2D Gagliardo-Nirenberg inequality, $\|f\|_{L^3(\Omega)}^3 \le C \|f\|_{L^2(\Omega)}^2 \|\nabla f\|_{L^2(\Omega)}$. 
Applying this to $\nabla w$, we have
\begin{equation*}
\int_{\Omega} |\nabla w|^3 \, dx\,dy \le C \|\nabla w\|_{L^2(\Omega)}^2 \|\Delta w\|_{L^2(\Omega)}.
\end{equation*}
Then we use Young's inequality to get
\begin{equation*}
C \|\nabla w\|_{L^2(\Omega)}^2 \|\Delta w\|_{L^2(\Omega)} \le \frac{1}{4} \|\Delta w\|_{L^2(\Omega)}^2 + C \|\nabla w\|_{L^2(\Omega)}^4.
\end{equation*}
So we have
\begin{equation}\label{global 2}
 -\int_{\Omega} w w_x (w_{xx} + w_{yy})\, dx\,dy \le \frac{1}{4} \|\Delta w\|_{L^2(\Omega)}^2 + C\|\nabla w\|_{L^2(\Omega)}^4.
\end{equation}
Second and third term together can be bounded as follows:
\begin{equation}\label{global 3}
-\int_{\Omega} (fw)_x (-\Delta w)\, dx\,dy = -\int_{\Omega} f_{xx}ww_x\, dx\,dy - \frac{3}{2}\int_{\Omega} f_xw_x^2\, dx\,dy - \frac{1}{2}\int_{\Omega} f_xw_y^2\, dx\,dy \le C \|w\|_{H^1(\Omega)}^2.
\end{equation}
Using the fact that the operator \(\partial^{-1}_x\partial_{yy}\) is skew symmetric, the second last term becomes
\begin{equation}\label{global 4}
 \int_{\Omega} \partial^{-1}_x w_{yy}\ \Delta w\, dx\,dy =0.
\end{equation}
Similarly the last term can be bounded as follows:
\begin{equation}\label{global 5}
\int_{\Omega} F \Delta w\, dx\, dy \le C \|F\|^2_{L^2(\Omega)} +\frac{1}{6} \|\Delta w\|^2_{L^2{(\Omega})}.
\end{equation}
Combining \eqref{global 2} -- \eqref{global 5}, we get the following:
\begin{equation}
\frac{d}{dt} \|\nabla w\|_{L^2(\Omega)}^2 +  \|\Delta w\|_{L^2(\Omega)}^2 \le C ((\|\nabla w\|_{L^2(\Omega)}^2 +1)  \|\nabla w\|_{L^2(\Omega)}^2) +C.
\end{equation}
Using \eqref{eq:grad_int} and applying Gronwall's inequality, we get
\begin{equation}\label{gloabl gradient sup}
    \sup_{t \in [0,T_0]} \|\nabla w(t)\|_{L^2(\Omega)}^2 \le C(T_0) 
\end{equation}
for any interval $[0,T_0]$. 

Equation \eqref{gloabl gradient sup} together with \eqref{eq:L2bound} gives us \eqref{H^1 bound}. 
\end{proof}

Combining Lemma \ref{local existence} and \ref{global exist}, we obtain global-in-time existence via a standard continuation argument. Therefore, for any $u^0$ with $u^0 - f \in H^s(\Omega)$ for $s>1$, there exists a unique global-in-time solution $u \in X_T$. This completes the proof of Theorem \ref{gloabal existence}

\section{Multi-d KdV/Zakharov–Kuznetsov Burgers equation}
In this section, we show that our method can be used to prove the $L^2$-contraction property for the following multi-dimensional Korteweg–de Vries (KdV)–Burgers equation, also known as the Zakharov–Kuznetsov–Burgers (ZKB) equation.
\begin{eqnarray}\label{KdVzk main equation}
    u_{t} + u u_{x_1} - \varepsilon \Delta u + \delta u_{x_1x_1x_1} + \kappa \sum_{i=2}^n u_{x_1x_ix_i} = 0, \qquad u(0,x) = u_{0}(x), \nonumber \\
    u(x_1,\ldots,x_i,\ldots,x_n)=u(x_1,\ldots,x_i+1,\ldots,x_n) \quad \forall i \in \{2,3,...,n\},
\end{eqnarray}
where \(u = u(t,x) \in \mathbb{R}\) is the unknown function, defined for \(t > 0\) and
\(x = (x_{1}, x')\) with \(x_{1} \in \mathbb{R}\) and
\(x' \in \mathbb{T}^{n-1} := \mathbb{R}^{n-1} / \mathbb{Z}^{n-1}\), \(n \geq 2\) being the \(n-1\) dimensional flat torus.

As before, if we consider a planar traveling viscous–dispersive shock wave propagating in the $x_1$-direction, described by a profile of the form: \(\tilde{u} (\xi) = \tilde{u} (x_1 - \sigma t)\) for (\ref{KdVzk main equation}), then $\tilde{u}$ satisfies \eqref{shock equation}. 

The main results can be stated as follows:
\begin{theorem}[Monotone Case]\label{main theorem for monotonezk}
    Let $\varepsilon$ and $\delta$ be positive constants, and let $\kappa$  be an arbitrary constant. Assume $ \frac{\d(u_--u_+)}{2\e^2}\leq \frac{1}{4} $, i.e., the traveling wave solution \(\tilde{u}\) of \eqref{shock equation}  is a monotonically decreasing solution connecting $u_{-}\ to\ u_{+}$. Let \(u^{0}\) be a given initial data satisfying \(u^{0}-\tilde{u}\in  L^2(\Omega)\), and let \(u\) be a solution to equation \eqref{KdVzk main equation}. If \beq\label{mono_asszk} \varepsilon > \frac{s}{4\pi}, \eeq then there exists a Lipschitz shift \(X(t)\) such that\\
\begin{equation}
\label{theorem1zk}
\begin{aligned}
\int_{\Omega} \big| u(t,x) - \tilde{u}(x_1 - \sigma t - X(t)) \big|^2 \,dx 
+ C_{*}\varepsilon \int_0^T \int_{\Omega} \left (\left( u(t,x)-\tilde{u}(x_1 - \sigma t - X(t)) \right)_{x_1} \right)^2 \,dx \,dt \\
+ C^{*} \sum_{i=2}^n \int_0^T\int_{\Omega} \left (\left(u(t,x)- \tilde{u}(x_1 - \sigma t - X(t)) \right)_{x_i} \right)^2 \, dx \, dt
\leq \int_{\Omega} |u^0 - \tilde{u}|^2 \, dx,
\end{aligned}
\end{equation}
where \(C_*= 2-\frac{1}{2\lambda_1}\), \(C^*=2\varepsilon-\frac{s^2}{8\pi^2\varepsilon} \),  and the Lipschitz shift satisfies the ODE,
\begin{equation} 
    \label{lipstichzk}
\begin{cases}
{\dot{X} = -\frac{M}{2s}\int_{\Omega}\Big( u(t,x_1+ X(t),x') - \tilde{u} (x_1 - \sigma t )\Big)\tilde{u}'(x_1 - \sigma t)\, dx_1 \, dx',}\\ {X(0) = 0,}
\end{cases} 
\end{equation}  
with $M=\frac{1}{2}$. Moreover, for $2<p<\infty$, the following holds:
\begin{equation} \label{tstabilitymonotonezk}
\int_{\Omega}
\left|u(t,x)-\tilde{u}(x_1-\sigma t-X(t))\right|^p \, dx 
\to 0 \quad \text{as } t \to \infty,
\end{equation}
which shows time-asymptotic stability.
Furthermore, the shift function satisfies the following time asymptotic behavior
\begin{equation} \label{shiftasyzk}
|\dot{X}(t)| \to 0, \quad \text{as } t \to \infty, \quad \text{and} \quad \frac{X(t)}{t} \to 0 \quad \text{as } t\to\infty. 
\end{equation} 
\end{theorem}

\begin{theorem}[Non-monotone Case] \label{main theorem for non monotonezk}
       Let \(\e=1\), \(\d\)  be a positive constant, and $\kappa$ be an arbitrary constant such that
\begin{equation} \label{cond1_2zk}
\frac{1}{4} < \d s < \frac{1}{2},
\end{equation}
and where \(u_\pm=\mp s\) be given states of far field with $s>0$.

 Let $u^0$ be given initial data satisfying  \(u^{0}-\tilde{u} \in  L^2(\Omega)\), and let \(u\) be any solution to equation \eqref{KdVzk main equation} with initial condition $u^0$. Then, for \beq\label{s_defzk} s < \sqrt{\frac{7.86}{M}} \pi ,\qquad\hbox{with}\qquad M = \frac{4}{3},\eeq
 the following $L^2$-contraction holds:
\begin{equation}\label{theorem2zk}
\begin{aligned}
\int_{\Omega} \big| u(t,x) - \tilde{u}(x_1 - \sigma t - X(t)) \big|^2 \, dx
+ \frac{1}{5} \int_0^T \int_{\Omega}\left(\left( u(t,x)-\tilde{u}(x_1 - \sigma t - X(t)) \right)_{x_1} \right)^2 \, dx \,dt \\
+ C^{**} \sum_{i=2}^n \int_0^T\int_{\Omega} \left (\left( u(t,x)-\tilde{u}(x_1 - \sigma t - X(t))\right)_{x_i} \right)^2 \, dx \, dt
\leq \int_{\Omega} |u^0 - \tilde{u}|^2 \, dx, \quad 
\end{aligned}
\end{equation}
where \(C^{**}=2-\frac{2.034 }{8\pi^2}  Ms^2 \), and the shift function $X(t)$ satisfies $X(0) = 0$,
\begin{equation}\label{shift non monotonezk}
\dot{X}(t) = -\frac{M}{2s} \int_{\Omega} \left( u(t, x + X(t),y) - \tilde{u}(x - \sigma t) \right) \tilde{u}'(x - \sigma t)\, dx_1\,dx'.
\end{equation}
Moreover, for $2<p<\infty$, the following holds:
\begin{equation} \label{tstabiltynonmonotonezk}
\int_{\Omega}
\left|u(t,x)-\tilde{u}(x_1-\sigma t-X(t))\right|^p \, dx 
\to 0 \quad \text{as } t \to \infty,
\end{equation}
which shows the time-asymptotic stability. Furthermore, the following time asymptotic behavior of the shift function holds:
\begin{equation} \label{shiftasy2zk}
|\dot{X}(t)| \to 0, \quad \text{as } t \to \infty ,\quad \text{and} \quad \frac{X(t)}{t} \to 0 \quad \text{as } t\to\infty.
\end{equation} 
\end{theorem}

Note the condition in \eqref{s_defzk} involves $\varepsilon$ if we do not assume $\varepsilon=1$.

\subsection{The main equation}
 We change the variable from $x_1$ to \(\xi = x_1 - \sigma t\) and rewrite (\ref{KdVzk main equation}) in terms of $u(t,\xi,x')$ as
\begin{equation}
\label{changezk of varibale equation}
u_t -\sigma u_\xi + u u_\xi= \varepsilon \Delta u - \delta u_{\x\x\x} - \kappa \sum_{i=2}^n u_{\x x_ix_i}
\end{equation}
where the operator \(\Delta\) is on \((\xi,x')\) variables.

\begin{lemma}\label{d_dtzklemma}
    Let $\tilde{u}$ be a planar viscous-dispersive shock satisfying equation (\ref{shock equation}), and for any $T>0$, let $u$ be any solution of equation \eqref{KdVzk main equation}. Then for any Lipschitz curve $X : [0,T] \to \mathbb{R}$, the following holds:
    \beq \label{main equation_0zk}
       \frac{d}{dt} \int_{\Omega} \frac{|u - \tilde{u}^{-X}|^2}{2} \,d\xi\, dx'
        =\dot{X} \int_{\Omega}(u^{X} - \tilde{u})\,\tilde{u}'\,d\xi\, dx' - \frac{1}{2}\int_{\Omega}|u^{X}-\tilde{u}|^2\,\tilde{u}' \,d\xi\,dx'- \varepsilon \int_{\Omega}|\nabla(u^{X} -\tilde{u})|^{2} \,d\xi\, dx',
\eeq
    where \(u^{X}(t,\xi ,y) = u(t,\xi +X(t),y) \) and \(\tilde{u}^{- X} = \tilde{u} (\xi -X(t))\).
\end{lemma}
\begin{proof}

We use \eqref{changezk of varibale equation} and \eqref{shock equation} to obtain 
\begin{align}
(u - \tilde{u}^{-X})_t 
&= \dot{X}(t)\,\tilde{u}_{\xi}^{-X} 
+ \sigma (u - \tilde{u}^{-X})_{\xi} 
- \left( \frac{u^2}{2} - \frac{(\tilde{u}^{-X})^2}{2} \right)_{\xi} + \varepsilon\, \Delta (u - \tilde{u}^{-X}) \nonumber\\
&\quad {\color{white}..........................................................}  
 - \delta\, (u - \tilde{u}^{-X})_{\xi \xi \xi}- \kappa \sum_{i=2}^n (u - \tilde{u}^{-X})_{\x x_i x_i}
\end{align}
Now we compute \\ \\
\noindent\makebox[\textwidth][l]{%
$\displaystyle
\frac{d}{dt} \int_{\Omega} \frac{|u - \tilde{u}^{-X}|^2}{2} \, d\xi\, dx'
$}
\begin{eqnarray} \nonumber
&=& \int_{\Omega} (u - \tilde{u}^{- X}) (u - \tilde{u} ^{- X})_t \, d\xi\, dx'  \nonumber \\
&=& \dot{X} \int_{\Omega} (u - \tilde{u}^{-X})\,\tilde{u}^{-X}_{\xi} \,d\xi\, dx' - \int_{\Omega} (u-\tilde{u}^{-X}) \left( \frac{u^2}{2} - \frac{(\tilde{u} ^{- X})^2}{2} \right)_\xi \,d\xi \,dx' \nonumber \\
&& + \varepsilon \int_{\Omega} (u - \tilde{u}^{-X})\,\Delta(u - \tilde{u}^{-X}) \,d\xi\, dx' - \delta \int_{\Omega} (u - \tilde{u}^{-X})(u - \tilde{u}^{-X})_{\xi\xi\xi} \, d\xi \,dx' \nonumber \\
&&  + \sigma \int_{\Omega} (u -\tilde{u}^{-X})_{\xi} \,(u -\tilde{u}^{-X})\, d\xi \,dx'-\kappa \sum_{i=2}^n\int_{\Omega} (u - \tilde{u}^{-X})(u - \tilde{u}^{-X})_{\x x_ix_i} \, d\xi \,dx'  \nonumber
\end{eqnarray}

We analyze each term on the right-hand side separately as follows:
\begin{equation*}
\begin{aligned}
-\int_{\Omega} (u - \tilde{u} ^{- X}) \left( \frac{u^2}{2} - \frac{(\tilde{u} ^{- X})^2}{2} \right)_\xi\, d\xi\, dx' 
&= \int_{\Omega} (u - \tilde{u} ^{- X})_\xi \left( \frac{u^2}{2} - \frac{(\tilde{u}^{ - X})^2}{2} \right)\, d\xi\, dx' \\
&= \frac{1}{2} \int_{\Omega} (u - \tilde{u}^{ - X})_\xi \left[ (u - \tilde{u} ^{- X})^2 + 2(\tilde{u}^ {- X})(u - \tilde{u} ^{- X}) \right]\, d\xi\, dx' \\
&= \frac{1}{2} \int_{\Omega} \left( (u - \tilde{u} ^{- X})^2 \right)_\xi (\tilde{u} ^{- X})\, d\xi\, dx' \\
&= -\frac{1}{2} \int_{\Omega} (u - \tilde{u} ^{- X})^2 (\tilde{u}')^{-X}\, d\xi\, dx'.
\end{aligned}
\end{equation*}

Using integration by parts and boundary conditions, we can prove that the dispersion term vanishes as:
\[\delta \int_{\Omega} (u - \tilde{u}^{-X})(u - \tilde{u}^{-X})_{\xi\xi\xi}\, d\xi\, dx' =\kappa \int_{\Omega} (u - \tilde{u}^{-X})(u - \tilde{u}^{-X})_{\x x_ix_i} \, d\xi \,dx'=0,\quad \forall\ i \in \{2,....,n\}.\]
Similarly, the dissipation term becomes:
\[ \varepsilon \int_{\Omega} (u - \tilde{u}^{-X})\,\Delta(u - \tilde{u}^{-X})\, d\xi\, dx' = - \varepsilon \int_{\Omega} |\nabla(u - \tilde{u}^{-X})|^2\, d\xi\, dx'. \]
The last term becomes:
\[\sigma \int_{\Omega} (u -\tilde{u}^{-X})_{\xi} \,(u -\tilde{u}^{-X}) \,d\xi\, dx' = \frac{\sigma}{2} \int_{\Omega} \left(|u -\tilde{u}^{-X}|^{2}\right)_{\xi} \, d\xi\, dx' =0. 
\]
Then, we use a change of variable \(\xi \mapsto \xi +X(t)\) to have
\begin{eqnarray} \label{last term mainzk}
  && \frac{d}{dt} \int_{\Omega} \frac{|u - \tilde{u}^{-X}|^2}{2} \,d\xi\, dx' \nonumber \\
   &=& \dot{X} \int_{\Omega}(u^{X} - \tilde{u})\,\tilde{u}' \,d\xi\, dx' - \frac{1}{2}
\int_{\Omega}|u^{X}-\tilde{u}|^2\,\tilde{u}'\,d\xi\, dx' - \varepsilon \int_{\Omega}|\nabla(u^{X} - \tilde{u})|^{2} \, d\xi \,dx' \nonumber
\end{eqnarray} 
This completes the proof of Lemma \ref{d_dtzklemma}.
\end{proof}
\bigskip

Then, using very similar methods as for the KP-equation, we can prove the main theorem for the ZKB equation. To make this paper self-contained, we add the proof of Theorem \ref{main theorem for monotonezk} for the monotonic case. But we leave the proof of Theorem \ref{main theorem for non monotonezk} for the non-monotonic case to the reader.

\subsection{Proof of Theorem \ref{main theorem for monotonezk}}

Now we choose $\dot X= -\frac{M}{2s} \left(\int_{\Omega}(u^{X} - \tilde{u})\,\tilde{u}'\,d\xi\, dx' \right)$ as in \eqref{lipstichzk} or \eqref{shift non monotonezk} with different $M$ for monotone and non-monotone cases, then we have 
\begin{eqnarray} \label{main equationzk}
  && \frac{d}{dt} \int_{\Omega} \frac{|u - \tilde{u}^{-X}|^2}{2} \,d\xi\, dx' \\
    &=&- \frac{M}{2s} \left(\int_{\Omega}(u^{X} - \tilde{u})\,\tilde{u}'\,d\xi\, dx' \right)^2 - \frac{1}{2}\int_{\Omega}|u^{X}-\tilde{u}|^2\,\tilde{u}' \,d\xi\, dx'- \varepsilon \int_{\Omega}|\nabla(u^{X} -\tilde{u})|^{2}  \,d\xi\, dx'. \nonumber
\end{eqnarray} 

Note, when $u_-=-u_+=s>0$, $\sigma=0$, so $\xi=x_1$.
We can rewrite (\ref{main equationzk}) with respect to the following variables
\[w=(u^{X} - \tilde{u})\circ(\x,x')  , \quad z = (u_{-} - \tilde{u}),\quad W = (u^{X} - \tilde{u})\circ (z,x').\]
Since \(\tilde{u}'< 0\), the change of variable \(\x\mapsto z\in [0,2s]\) satisfying
\begin{equation}
\label{dzzk}
  \frac{dz}{d\xi} = -\tilde{u}' \quad  
\end{equation}
is well-defined for $\xi\in (-\infty,\infty)$.

Since it holds from (\ref{lipstichzk}) that
\[\dot{X} = -\frac{1}{4s}\int_{\mathbb{T}^{n-1}}\int_{0}^{2s}(u(t,\xi +X(t),x') - \tilde{u}(\xi))\, \tilde{u}'(\xi)\,d\xi\, dx'.\]
From (\ref{dzzk}), we get
\begin{equation} \label{X in terms of W barzk}
    \dot{X} = \frac{1}{4s}\int_{\mathbb{T}^{n-1}}\overline{W} \,dx',\qquad \int_{\Omega}(u^{X} - \tilde{u})\,\tilde{u}' \,d\xi\, dx' = - \int_{\mathbb{T}^{n-1}}\overline{W} \,dx', \quad 
\end{equation}
where \(\overline{W}= \int_{0}^{2s}W \,dz\). 

For the second term in (\ref{main equationzk})
\beq \label{B boundzk}
-\frac{1}{2}\int_{\Omega}|u^{X}-\tilde{u}|^2\, \tilde{u}'\,d\xi \, dx' 
= \frac{1}{2}\int_{\mathbb{T}^{n-1}}\int_{0}^{2s}W^{2}\,dz\,dx'.
\eeq
From 
(\ref{main lemma}), we get
\begin{equation} \label{lower bound of dzzk}
 \frac{\lambda_1 z(2s-z)}{\varepsilon} \le \frac{dz}{d\xi} \le \frac{\gamma z(2s-z)}{\varepsilon}.
\end{equation}
For the third term in (\ref{main equationzk}), by using (\ref{lower bound of dzzk}), we get
\begin{eqnarray} \label{G boundzk}
 \varepsilon\int_{\mathbb{T}^{n-1}}\int_{\mathbb{R}}|\nabla(u^{X} - \tilde{u})|^{2}\,d\xi\, dx'  
&=&\varepsilon\int_{\mathbb{T}^{n-1}}\int_{0}^{2s}|\partial_{z}W|^{2}\frac{dz}{d\xi} \,dz\, dx'  + \varepsilon\sum_{i=2}^n\int_{\mathbb{T}^{n-1}}\int_{\mathbb{R}}|\partial_{x_i}w|^{2}\, d\x \,dx' \nonumber\\
&\geq& \lambda_1 \int_{\mathbb{T}^{n-1}}\int_{0}^{2s}z(2s - z)|W_{z}|^{2} \,dz\, dx'   + \varepsilon \sum_{i=2}^n  \int_{\mathbb{T}^{n-1}}\int_{\mathbb{R}}|\partial_{x_i}w|^{2} \,d\x\, dx'. \nonumber \\
\end{eqnarray}

Using (\ref{X in terms of W barzk}), (\ref{B boundzk}), (\ref{G boundzk}), we get
\begin{eqnarray}
&&
   - \frac{1}{4s} \left(\int_{\Omega}(u^{X} - \tilde{u})\tilde{u}' \,d\xi\, dx' \right)^2 - \frac{1}{2}\int_{\Omega}|u^{X}-\tilde{u}|^2\tilde{u}' \,d\xi\,dx'- \varepsilon \int_{\Omega}|\nabla(u^{X} - \tilde{u})|^{2} \, d\xi \,dx'  \nonumber\\
&\leq& -\frac{1}{4s}\left(\int_{\mathbb{T}^{n-1}}\overline{{W}} \,dx'\right)^{2}+\frac{1}{2}\int_{\mathbb{T}^{n-1}}\int_{0}^{2s}W^{2}\,dz\, dx' -\lambda_1 \int_{\mathbb{T}^{n-1}}\int_{0}^{2s}z(2s - z)|W_{z}|^{2} \,dz\, dx'  \nonumber \\
&& -  \varepsilon \sum_{i=2}^n  \int_{\mathbb{T}^{n-1}}\int_{\mathbb{R}}|\partial_{x_i}w|^{2} \,d\x\, dx'
  \nonumber \\
&=& -\frac{1}{4s}\left(\int_{\mathbb{T}^{n-1}}\overline{{W}}\,dx'\right)^{2}+\frac{1}{4s}\int_{\mathbb{T}^{n-1}}\overline{{W}}^{2}\,dx' \nonumber\\
&&+\frac{1}{2}\int_{\mathbb{T}^{n-1}}\left(\int_{0}^{2s}W^{2}\,dz-\frac{1}{2s}\overline{{W}}^{2}-\int_{0}^{2s}\frac{1}{2}z(2s-z)|W_{z}|^{2}\,dz\right)\,dx'  \nonumber\\
&& +\left(\frac{1}{4}-\lambda_1\right)\int_{\mathbb{T}^{n-1}} \int_{0}^{2s}z(2s-z)|W_{z}|^{2}\,dz\,\, dx'\,  -\varepsilon \sum_{i=2}^n  \int_{\mathbb{T}^{n-1}}\int_{\mathbb{R}}|\partial_{x_i}w|^{2} \,d\x\, dx'.
 \quad  \nonumber 
\end{eqnarray}
Note that from (\ref{poincare0}), we have

\[ \int_{0}^{2s}W^2 \,dz- \frac{1}{2s}\overline{{W}}^{2} - \frac{1}{2}\int_{0}^{2s}z(2s - z)|W_{z}|^{2}\,dz\leq 0.\]
Using equation \eqref{G boundzk}, together with Poincaré inequality on $\mathbb{T}^{n-1}$, we obtain a bound for the right hand side of \eqref{main equationzk} as follows:
\begin{eqnarray}
  &&  \frac{-1}{4s} \left(\int_{\Omega}(u^{X} - \tilde{u})\tilde{u}'\,d\xi\, dx' \right)^2 - \frac{1}{2}\int_{\Omega}|u^{X}-\tilde{u}|^2\tilde{u}' \,d\xi\,dx'- \varepsilon \int_{\Omega}|\nabla(u^{X} - \tilde{u})|^{2} \, d\xi\, dx' \nonumber \qquad\qquad
  \end{eqnarray}
  
  \begin{align}
&\leq -\frac{1}{4s}\left(\int_{\mathbb{T}^{n-1}}\overline{{W}}\,dx'\right)^{2}+\frac{1}{4s}\int_{\mathbb{T}^{n-1}}\overline{{W}}^{2}\,dx' +\left(\frac{1}{4}-\lambda_1\right)\int_{\mathbb{T}^{n-1}} \int_{0}^{2s}z(2s-z)|W_{z}|^{2} \,dz \,dx'  \nonumber \\
& - \varepsilon \sum_{i=2}^n  \int_{\mathbb{T}^{n-1}}\int_{\mathbb{R}}|\partial_{x_i}w|^{2} \,d\x\, dx' 
 \quad \nonumber\\
&\leq \frac{1}{4s}\int_{\mathbb{T}^{n-1}} \left(\overline{W} - \int_{\mathbb{T}^{n-1}} \overline{W}\right)^2 \,dx' -\varepsilon \sum_{i=2}^n  \int_{\mathbb{T}^{n-1}}\int_{\mathbb{R}}|\partial_{x_i}w|^{2} \,d\x\, dx' 
 +\left(\frac{\varepsilon}{4\lambda_1}-\varepsilon\right)\int_{\mathbb{T}^{n-1}} \int_{\mathbb{R}}|w_{\xi}|^{2}\,d\xi\, dx'
 \quad  \nonumber \\
&\le\frac{1}{16s\pi^2}  \sum_{i=2}^n \int_{\mathbb{T}^{n-1}} |\overline{W}_{x_i}|^2 dx'  -\varepsilon \sum_{i=2}^n  \int_{\mathbb{T}^{n-1}}\int_{\mathbb{R}}|\partial_{x_i}w|^{2} \,d\x\, dx' 
+\left(\frac{\varepsilon}{4\lambda_1}-\varepsilon\right)\int_{\mathbb{T}^{n-1}} \int_{\mathbb{R}}|w_{\xi}|^{2}\,d\xi\, dx'
 \quad  \nonumber\\
&\le \frac{1}{16s\pi^2} \Big( \sum_{i=2}^n \int_{\mathbb{T}^{n-1}} \int_{\mathbb{R}}|\partial_{x_i}{w}|^2\, d\xi\, dx' \Big ) \Big( \int_{\mathbb{R}} |\tilde{u}'|^2 \, d\xi\Big) -\varepsilon \sum_{i=2}^n  \int_{\mathbb{T}^{n-1}}\int_{\mathbb{R}}|\partial_{x_i}w|^{2} \,d\x\, dx' \nonumber \\ 
&+\left(\frac{\varepsilon}{4\lambda_1}-\varepsilon\right)\int_{\mathbb{T}^{n-1}} \int_{\mathbb{R}}|w_{\xi}|^{2}\,d\xi \,dx'
 \quad  \nonumber \\
&\le \frac{\sup_{\mathbb{R}} |\tilde{u}'|}{16s\pi^2} \Big( \sum_{i=2}^n \int_{\mathbb{T}^{n-1}} \int_{\mathbb{R}}|\partial_{x_i}{w}|^2 \,d\xi \,dx' \Big ) \Big( \int_{\mathbb{R}} |\tilde{u}'| \, d\xi\Big) -\varepsilon \sum_{i=2}^n  \int_{\mathbb{T}^{n-1}}\int_{\mathbb{R}}|\partial_{x_i}w|^{2} \,d\x\, dx'\nonumber \\
& +\left(\frac{\varepsilon}{4\lambda_1}-\varepsilon\right)\int_{\mathbb{T}^{n-1}} \int_{\mathbb{R}}|w_{\xi}|^{2} \,d\xi \,dx'
 \quad  \nonumber \\
&\le\left(\frac{s^2}{16\pi^2\varepsilon}- \varepsilon\right) \sum_{i=2}^n \int_{\mathbb{T}^{n-1}} \int_{\mathbb{R}} |\partial_{x_i}{w}|^2 \,d\x \,dx'  +\left(\frac{\varepsilon}{4\lambda_1}-\varepsilon\right)\int_{\mathbb{T}^{n-1}} \int_{\mathbb{R}}|w_{\xi}|^{2} \,d\xi\, dx' \label{less than zerozk}
\end{align}
where we used \eqref{total variation for monotone} in the last step. From \eqref{mono_asszk}, we have \(C^* = 2\varepsilon - \frac{s^2}{8\pi^2 \varepsilon} > 0\). Moreover, since \(\lambda_1 > \frac{1}{4}\), it follows that
\(
C_* = 2 - \frac{1}{2\lambda_1} > 0
\). Using \eqref{main equationzk} and \eqref{less than zerozk}, we obtain
\begin{align}
\frac{d}{dt} \int_{\Omega} \frac{|u - \tilde{u}^{-X}|^2}{2}\, d\x \, dx'
&+ \frac{C_{*}\varepsilon}{2} \int_{\Omega} \left( \partial_{x_1} \big(u(t,x)-\tilde{u}(x  - X(t))\big) \right)^2 \, dx_1\,dx' \nonumber\\
&+ \frac{C^{*}}{2} \sum_{i=2}^n \int_{\Omega} \left( \partial_{x_i} \big(u(t,x)-\tilde{u}(x - X(t))\big) \right)^2 \, dx_1\,dx' \label{final equation monotone casezk}
\le 0.
\end{align}
By integrating \eqref{final equation monotone casezk} over the interval [0,T], we obtain \eqref{theorem1zk}.

Using \eqref{theorem1zk}, the proof of \eqref{tstabilitymonotonezk} is analogous to that of Theorem 2.3 in \cite{Hur}. Consequently, \eqref{tstabilitymonotonezk} implies that 
\begin{equation}\nonumber\label{Xdotestimatezk}
|\dot{X}(t)|
\leq \| u(t,\cdot,\cdot) - \tilde{u}(\cdot - X(t)) \|_{L^4(\Omega)}
 \|\tilde{u}' \|_{L^\frac{4}{3}(\mathbb{R})}
\;\longrightarrow\; 0
\quad \text{as } t \to \infty,
\end{equation}
since \( \|\tilde{u}'\|_{L^\frac{4}{3}(\mathbb{R})}\) is bounded.
This completes the proof of Theorem \ref{main theorem for monotonezk}.

\appendix
\section{Proof of Theorem \ref{induction step thm}}
\label{app}
Now we add the proof of the Theorem \ref{induction step}, which is very similar to the one in Section 4 of \cite{chen2026uniform}. We add it to make this paper self-contained.
\begin{proof}
Recall we use the following notations: we index the local extrema of \(\util(\x)\) from right to left.
The rightmost extremum is denoted by \(u_0\), and the remaining extrema are indexed sequentially as \(u_1, \ u_2, \cdots\) in the order they appear when moving leftward.
We write \(\x_i\) for the spatial location at which the value \(u_i\) is attained, so that \(\util(\x_i)=u_i\), as shown in Figure \ref{pp}.

Next we denote
\begin{equation} \label{badgood}
\Hcal_i \coloneqq -\frac{1}{2}\int_{\mathbb{T}^1}\int_{\x_i}^{\x_{i-1}} w^2 \util' \,d\x\, dy, \qquad
\Dcal \coloneqq \int_{\mathbb{T}^1} \int_\RR (w_\x)^2 \,d\x\, dy, \qquad
\Dcal[a,b] \coloneqq \int_{\mathbb{T}^1} \int_a^b (w_\x)^2 \,d\x\,dy.
\end{equation}
We introduce the change of variable on each monotonic interval \((\x_i,\x_{i-1})\) as follows:
\begin{equation} \label{ztu0}
z(\x) \coloneqq \util(\x), \qquad
dz = \util' \,d\x.
\end{equation}
We use the following in the upcoming arguments,
\begin{equation} \label{transfer0}
\int_{\mathbb{T}^1}\Big(\int_{J_i} w\util' \,d\x\Big)^2 \,dy
=\int_{\mathbb{T}^1}\Big(\int_{J_i} w_\x (\util-u_i) \,d\x\Big)^2 \,dy
\le \Big(\int_{J_i} (\util-u_i)^2 \,d\x\Big) 
\Dcal[\x_i,\x^i]  ,
\end{equation}
where \(J_i=(\x_i,\x^i)\).

\step{0}
First of all, using the same argument as in \eqref{transfer0}, we observe 
\[
-\Big(\int_{-\infty}^{\x_s}(\util-s)^2\,d\x\Big) \Big(\int_{\mathbb{T}^1}\int_{-\infty}^{\x_s} (w_\x)^2 \,d\x \,dy\Big)
\le -\int_{\mathbb{T}^1}\Big(\int_{-\infty}^{\x_s} w\util' \,d\x\Big)^2 \,dy.
\]
Then, using \eqref{L2B-S} in Proposition \ref{prop:L2}, we find that
\begin{equation} \label{step0-1}
-20 C_0 (0.001) \Dcal[-\infty,\x_s]
\le -\frac{40C_0}{2s} \int_{\mathbb{T}^1}\Big(\int_{-\infty}^{\x_s} w\util' \,d\x\Big)^2 \,dy,
\end{equation}
where \(C_0\) is a positive constant to be determined at the end of \textit{Step 0}.
Then, we get
\[
-\frac{M}{2s}\int_{\mathbb{T}^1} \Big(\int_\RR w\util' \,d\x\Big)^2 \,dy
-\frac{40C_0}{2s}\int_{\mathbb{T}^1} \Big(\int_{-\infty}^{\x_s} w\util' \,d\x\Big)^2 \,dy
\le -\frac{1}{2s}\frac{40 C_0 M}{40C_0 + M} \int_{\mathbb{T}^1} \Big(\int_{\x_s}^\infty w\util' \,d\x\Big)^2 \,dy.
\]
Thus, for large enough \(M\)---in particular, for any \(M\ge\frac{40}{39}C_0\)---we have
\begin{equation} \label{step0-2}
-\frac{M}{2s} \int_{\mathbb{T}^1}\Big(\int_\RR w\util' \,d\x\Big)^2 \,dy
-\frac{40C_0}{2s}\int_{\mathbb{T}^1} \Big(\int_{-\infty}^{\x_s} w\util' \,d\x\Big)^2 \,dy
\le -\frac{C_0}{2s}\int_{\mathbb{T}^1}\Big(\int_{\x_s}^\infty w\util' \,d\x\Big)^2 \,dy.
\end{equation}
We also exploit a portion of \(\Hcal_1\) as follows:
\begin{equation} \label{step0-3}
\frac{1}{15}\Hcal_1
=-\frac{1}{30}\int_{\mathbb{T}^1}\int_{\x_1}^{\x_0} w^2\util' \,d\x \,dy
\le -\frac{1}{30} \int_{\mathbb{T}^1} \int_{\x_s}^{\x_0} w^2 \util' \,d\x \,dy
\le -\frac{1}{30(u_0-s)} \int_{\mathbb{T}^1} \Big(\int_{\x_s}^{\x_0} w\util' \,d\x\Big)^2 \,dy.
\end{equation}
Using \eqref{step0-2} and \eqref{step0-3}, we obtain the squared average term on the rightmost decreasing interval:
\begin{eqnarray} \label{step0-4}
-\frac{C_0}{2s}\int_{\mathbb{T}^1}\Big(\int_{\x_s}^\infty w\util' \,d\x\Big)^2 \,dy
-\frac{1}{30(u_0-s)} \int_{\mathbb{T}^1}\Big(\int_{\x_s}^{\x_0} w\util' \,d\x\Big)^2 \,dy \nonumber \\
\le -\frac{C_0}{30C_0(u_0-s)+2s} \int_{\mathbb{T}^1}\Big(\int_{\x_0}^\infty w\util' \,d\x\Big)^2 \,dy.
\end{eqnarray}
We choose the constant \(C_0\) so that it satisfies
\[
\frac{C_0}{30C_0(u_0-s)+2s} \ge \frac{7}{12(u_0+s)}, \quad \text{or equivalently,} \quad
C_0 \ge \frac{14s}{222s-198u_0}.
\]
On the other hand, since \(u_0\le1.0601s\) by \eqref{u0-exp}, it suffices to choose \(C_0\ge\frac{13}{10}\).
We then fix \(C_0=\frac{13}{10}\).
Based on the change of variables \eqref{ztu0}, we use Lemma \ref{KV}, with Proposition \ref{prop:key} and \eqref{step0-4}, to get
\begin{equation} \label{step0-5}
\begin{aligned}
&-\frac{7}{12(u_0+s)}\int_{\mathbb{T}^1}\Big(\int_{\x_0}^\infty w\util' \,d\x\Big)^2 \,dy
-\frac{7}{24}(\lbar_0)^{-1} \Dcal[\x_0,\infty]\\
&\qquad
\le -\frac{7}{12(u_0+s)}\int_{\mathbb{T}^1}\Big(\int_{-s}^{u_0} w \,dz\Big)^2 \,dy
-\frac{7}{24}\int_{\mathbb{T}^1}\int_{-s}^{u_0} (u_0-z)(z+s) (w_z)^2 \,dz \,dy\\
&\qquad
\le -\frac{7}{12} \int_{\mathbb{T}^1} \int_{-s}^{u_0} w^2 \,dz \,dy
=\frac{7}{12} \int_{\mathbb{T}^1} \int_{\x_0}^\infty w^2 \util' \,d\x \,dy.
\end{aligned}
\end{equation}
Note that \(\lbar_0=0.355\) and \(\frac{7}{24}(\lbar_0)^{-1} < 0.83\).
This cancels with \(\Hcal_0\) as follows:
\begin{eqnarray} \label{step0-6}
\frac{7}{12} \int_{\mathbb{T}^1} \int_{\x_0}^\infty w^2 \util' \,d\x \,dy
+\Hcal_0
&=&\frac{7}{12} \int_{\mathbb{T}^1} \int_{\x_0}^\infty w^2 \util' \,d\x \,dy
-\frac{1}{2}\int_{\mathbb{T}^1} \int_{\x_0}^\infty w^2 \util' \,d\x \,dy \nonumber \\
&=&\frac{1}{12} \int_{\mathbb{T}^1} \int_{\x_0}^\infty w^2 \util' \,d\x \,dy.
\end{eqnarray}
Thus, summing up \eqref{step0-1}--\eqref{step0-6}, we summarize \textit{Step 0} as follows:
\begin{eqnarray} \label{step0}
\begin{aligned}
-(0.026)\Dcal[-\infty,\x_s]
-\frac{M}{2s}\int_{\mathbb{T}^1}\Big(\int_\RR w\util' \,d\x\Big)^2 \,dy
+\frac{1}{15}\Hcal_1 
-(0.83)\Dcal[\x_0,\infty]
+\Hcal_0  \\
\le \frac{1}{12} \int_{\mathbb{T}^1} \int_{\x_0}^\infty w^2 \util' \,d\x \,dy.
\end{aligned}
\end{eqnarray}
We note that \(\Hcal_1\) is negative, whereas \(\Hcal_0\) is positive.
We also recall that \(C_0=\frac{13}{10}\) and \(M\ge\frac{4}{3}\). Fix \(M=\frac{4}{3}\).

\step{1} In this step, we verify the base case of the induction---we obtain the following good term: 
\begin{equation} \label{base}
-\Big(\frac{1}{2}+a_1\Big)\int_{\mathbb{T}^1}\int_{\x_1}^{\x_0} w^2 \util' \,d\x \,dy,
\end{equation}
for some constant \(a_1>0\).
To this end, using the remainder in \eqref{step0}, we first observe
\begin{equation} \label{step1-1}
\frac{1}{12}\int_{\mathbb{T}^1} \int_{\x_0}^\infty w^2 \util' \,d\x \,dy
\le \frac{1}{12} \int_{\mathbb{T}^1}\int_{\x_0}^{\x^1} w^2 \util' \,d\x \,dy
\le -\frac{1}{12(u_0-u_1)} \int_{\mathbb{T}^1} \Big(\int_{\x_0}^{\x^1} w \util' \,d\x\Big)^2 \,dy.
\end{equation}
On the other hand, using \eqref{transfer0} with \eqref{L2B} in Proposition \ref{prop:L2}, we find that
\begin{equation} \label{step1-2}
-0.178C_1 \int_{\mathbb{T}^1} \int_{J_1} (w_\x)^2 \,d\x \,dy
=-0.178C_1 \Dcal[\x_1,\x^1]
\le -\frac{C_1}{u_0-u_1} \int_{\mathbb{T}^1} \Big(\int_{J_1} w\util'\,d\x\Big)^2 \,dy,
\end{equation}
where \(C_1>0\) is a constant to be determined later.
The two estimates above show that we have
\[
-\frac{1}{12(u_0-u_1)} \int_{\mathbb{T}^1} \Big(\int_{\x_0}^{\x^1} w \util' \,d\x\Big)^2 \,dy
-\frac{C_1}{u_0-u_1} \int_{\mathbb{T}^1} \Big(\int_{J_1} w\util' \,d\x\Big)^2\, dy \]
\[\le - \frac{C_1}{(12C_1+1)(u_0-u_1)} \int_{\mathbb{T}^1} \Big(\int_{\x_1}^{\x_0} w \util'\, d\x\Big)^2 \,dy.
\]
We choose \(C_1\) to satisfy 
\[
\frac{C_1}{12C_1+1} = \frac{1}{15}, \quad \text{or equivalently,} \quad
C_1=\frac{1}{3}.
\]
Then, we rewrite the above inequality into the following:
\begin{eqnarray} \label{step1-3}
-\frac{1}{12(u_0-u_1)} \int_{\mathbb{T}^1}\Big(\int_{\x_0}^{\x^1} w \util' \,d\x\Big)^2 \,dy
-\frac{1}{3(u_0-u_1)}\int_{\mathbb{T}^1} \Big(\int_{J_1} w\util' \,d\x\Big)^2 \,dy \nonumber \\
\le - \frac{1}{15(u_0-u_1)}\int_{\mathbb{T}^1} \Big(\int_{\x_1}^{\x_0} w \util' \,d\x\Big)^2 \,dy.
\end{eqnarray}
It follows from Lemma \ref{KV} and Proposition \ref{prop:key} with \eqref{ztu0} that
\begin{equation} \label{step1-4}
\begin{aligned}
&- \frac{1}{15(u_0-u_1)} \int_{\mathbb{T}^1}\Big(\int_{\x_1}^{\x_0} w \util' \,d\x\Big)^2 \,dy
-\frac{1}{30}(\lbar_1)^{-1} \Dcal[\x_1,\x_0]\\
&\qquad
\le -\frac{1}{15(u_0-u_1)}\int_{\mathbb{T}^1}\Big(\int_{u_1}^{u_0} w \,dz\Big)^2 \,dy
-\frac{1}{30}\int_{\mathbb{T}^1}\int_{u_1}^{u_0} (u_0-z)(z-u_1) (w_z)^2 \,dz \,dy\\
&\qquad
\le -\frac{1}{15}\int_{\mathbb{T}^1} \int_{u_1}^{u_0} w^2 \,dz \,dy
=-\frac{1}{15}\int_{\mathbb{T}^1} \int_{\x_1}^{\x_0} w^2 \util' \,d\x \,dy.
\end{aligned}
\end{equation}
We recall \(\lbar_1=9.60\).
Thus, combining \eqref{step1-1}-\eqref{step1-4}, we obtain the following summary of \textit{Step 1}:
\[
\frac{1}{12}\int_{\mathbb{T}^1} \int_{\x_0}^\infty w^2 \util' \,d\x \,dy
-(0.06)\Dcal[\x_1,\x^1]
-(0.01)\Dcal[\x_1,\x_0]
\le -\frac{1}{15} \int_{\mathbb{T}^1}\int_{\x_1}^{\x_0} w^2 \util' \,d\x \,dy.
\]
This together with \eqref{step0} and the \(\Hcal_1\) term implies that
\begin{equation} \label{step0and1}
\begin{aligned}
&-\frac{M}{2s}\int_{\mathbb{T}^1}\Big(\int_\RR w\util' d\x\Big)^2 dy
+\Hcal_0
-(0.83)\Dcal[\x_0,\infty]
-(0.026)\Dcal[-\infty,\x_s]\\
&\qquad
-(0.06)\Dcal[\x_1,\x^1]
-(0.01)\Dcal[\x_1,\x_0]
+\Hcal_1
\le -\Big(\frac{1}{2}+\frac{1}{30}\Big)\int_{\mathbb{T}^1}\int_{\x_1}^{\x_0} w^2\util'\,d\x \,dy.
\end{aligned}
\end{equation}
This verifies the base case \eqref{base} with \(a_1=\frac{1}{30}\).

\step{2}
Similar to \cite{chen2026uniform} we introduce two sequences as follows: for each \(n\in\NN\),
\begin{equation} \label{seq}
a_n\coloneqq\Big(\frac{1}{15}\Big)2^{-n}, \qquad
C_{2n}\coloneqq a_{2n-1}+\frac{3}{2}+\frac{1}{2a_{2n-1}}, \qquad
C_{2n+1}\coloneqq a_{2n}.
\end{equation}
We now assume the induction hypothesis: for some odd \(i\ge1\), the following good term is available:
\begin{equation} \label{ind-hypo0}
-\Big(\frac{1}{2}+a_i\Big) \int_{\mathbb{T}^1} \int_{\x_i}^{\x_{i-1}} w^2 \util' \,d\x \,dy.
\end{equation}
We show that the induction hypothesis holds for \(i+2\).
To this end, we first observe
\begin{eqnarray} \label{step2-1}
-\Big(\frac{1}{2}+a_i\Big)\int_{\mathbb{T}^1} \int_{\x_i}^{\x_{i-1}} w^2 \util' \,d\x \,dy
&\le& -\Big(\frac{1}{2}+a_i\Big) \int_{\mathbb{T}^1} \int_{\x_i}^{\x^{i+1}} w^2 \util' \,d\x \,dy \nonumber \\
&\le&-\frac{\big(\frac{1}{2}+a_i\big)}{u_{i+1}-u_i} \int_{\mathbb{T}^1}\Big(\int_{\x_i}^{\x^{i+1}} w \util' \,d\x\Big)^2 \,dy.
\end{eqnarray}
Meanwhile, it follows from \eqref{transfer0} with \eqref{L2B} in Proposition \ref{prop:L2} that
\begin{equation} \label{step2-2}
\begin{aligned}
-(0.81)(\r_*)^{-(i+1)}C_{i+1}\int_{\mathbb{T}^1}\int_{J_{i+1}} (w_\x)^2 \,d\x \,dy
&=-(0.81)(\r_*)^{-(i+1)}C_{i+1}\Dcal[\x_{i+1},\x^{i+1}]\\
&\le -\frac{C_{i+1}}{u_{i+1}-u_i}\int_{\mathbb{T}^1}\Big(\int_{J_{i+1}} w\util'\,d\x\Big)^2 \,dy.
\end{aligned}
\end{equation}
Then, using \eqref{step2-1} and \eqref{step2-2}, we obtain
\begin{align*}
&-\frac{\big(\frac{1}{2}+a_i\big)}{u_{i+1}-u_i} \int_{\mathbb{T}^1}\Big(\int_{\x_i}^{\x^{i+1}} w \util' \,d\x\Big)^2 \,dy
-\frac{C_{i+1}}{u_{i+1}-u_i}\int_{\mathbb{T}^1}\Big(\int_{J_{i+1}} w\util'\,d\x\Big)^2 \,dy\\
&\qquad
\le -\frac{1}{u_{i+1}-u_i}\frac{\big(\frac{1}{2}+a_i\big)C_{i+1}}{\big(\frac{1}{2}+a_i\big)+C_{i+1}}
\int_{\mathbb{T}^1}\Big(\int_{\x_{i+1}}^{\x_i} w \util' \,d\x\Big)^2 \,dy.
\end{align*}
Thanks to \eqref{seq}, this is equivalent to the following:
\begin{eqnarray} \label{step2-3}
-\frac{\big(\frac{1}{2}+a_i\big)}{u_{i+1}-u_i} \int_{\mathbb{T}^1}\Big(\int_{\x_i}^{\x^{i+1}} w \util' \,d\x\Big)^2 \,dy
 -\frac{C_{i+1}}{u_{i+1}-u_i}\int_{\mathbb{T}^1}\Big(\int_{J_{i+1}} w\util'\,d\x\Big)^2 \,dy \nonumber \\
\le -\frac{\big(\frac{1}{2}+a_{i+1}\big)}{u_{i+1}-u_i} \int_{\mathbb{T}^1}\Big(\int_{\x_{i+1}}^{\x_i} w \util' \,d\x\Big)^2 \,dy.
\end{eqnarray}
We apply Lemma \ref{KV} together with Proposition \ref{prop:key} and \eqref{ztu0} to obtain
\begin{equation} \label{step2-4}
\begin{aligned}
&-\frac{\big(\frac{1}{2}+a_{i+1}\big)}{u_{i+1}-u_i} \int_{\mathbb{T}^1}\Big(\int_{\x_{i+1}}^{\x_i} w \util' \,d\x\Big)^2 \,dy
-\frac{1}{2}\Big(\frac{1}{2}+a_{i+1}\Big)(\lbar_{i+1})^{-1} \Dcal[\x_{i+1},\x_i]\\
&\qquad
\le -\frac{\big(\frac{1}{2}+a_{i+1}\big)}{u_{i+1}-u_i} \int_{\mathbb{T}^1}\Big(\int_{u_i}^{u_{i+1}} w \,dz\Big)^2 \,dy
-\frac{1}{2}\Big(\frac{1}{2}+a_{i+1}\Big) \int_{\mathbb{T}^1}\int_{u_i}^{u_{i+1}} (u_{i+1}-z)(z-u_i) (w_z)^2 \,dz \,dy\\
&\qquad
\le -\Big(\frac{1}{2}+a_{i+1}\Big)\int_{\mathbb{T}^1} \int_{u_i}^{u_{i+1}} w^2 \,dz \,dy
=\Big(\frac{1}{2}+a_{i+1}\Big) \int_{\mathbb{T}^1}\int_{\x_{i+1}}^{\x_i} w^2 \util' \,d\x \,dy,
\end{aligned}
\end{equation}
where \(\lbar_{i+1}=(1.94)(\r_*)^{i+1}\).
This cancels with \(\Hcal_{i+1}\) as follows:
\begin{equation} \label{step2-5}
\Big(\frac{1}{2}+a_{i+1}\Big) \int_{\mathbb{T}^1}\int_{\x_{i+1}}^{\x_i} w^2 \util' \,d\x \,dy
+\Hcal_{i+1}
=a_{i+1}\int_{\mathbb{T}^1}\int_{\x_{i+1}}^{\x_i} w^2 \util' \,d\x \,dy(\le 0).
\end{equation}
This yields the following good term: 
\begin{eqnarray} \label{step2-6}
a_{i+1}\int_{\mathbb{T}^1}\int_{\x_{i+1}}^{\x_i} w^2 \util' \,d\x \,dy
&\le& a_{i+1}\int_{\mathbb{T}^1}\int_{\x_{i+1}}^{\x^{i+2}} w^2 \util' \,d\x \,dy \nonumber\\
&\le& -\frac{a_{i+1}}{u_{i+1}-u_{i+2}} \int_{\mathbb{T}^1}\Big(\int_{\x_{i+1}}^{\x^{i+2}} w\util' \,d\x\Big)^2 \,dy.
\end{eqnarray}
Moreover, using \eqref{transfer0} with \eqref{L2B} in Proposition \ref{prop:L2}, we obtain
\begin{equation} \label{step2-7}
\begin{aligned}
-(0.82)(\r_*)^{-(i+2)}C_{i+2}\int_{\mathbb{T}^1}\int_{J_{i+2}} (w_\x)^2 \,d\x \,dy
&=-(0.82)(\r_*)^{-(i+2)}C_{i+2}\Dcal[\x_{i+2},\x^{i+2}]\\
&\le -\frac{C_{i+2}}{u_{i+1}-u_{i+2}}\int_{\mathbb{T}^1}\Big(\int_{J_{i+2}} w\util'\,d\x\Big)^2 \,dy.
\end{aligned}
\end{equation}
Combining \eqref{step2-6} and \eqref{step2-7}, we find that
\begin{align*}
&-\frac{a_{i+1}}{u_{i+1}-u_{i+2}} \int_{\mathbb{T}^1}\Big(\int_{\x_{i+1}}^{\x^{i+2}} w\util' \,d\x\Big)^2 \,dy
-\frac{C_{i+2}}{u_{i+1}-u_{i+2}}\int_{\mathbb{T}^1}\Big(\int_{J_{i+2}} w\util'\,d\x\Big)^2 \,dy\\
&\qquad
\le -\frac{1}{u_{i+1}-u_{i+2}}\frac{a_{i+1}C_{i+2}}{a_{i+1}+C_{i+2}}
\int_{\mathbb{T}^1}\Big(\int_{\x_{i+2}}^{\x_{i+1}}w\util' \,d\x\Big)^2 \,dy.
\end{align*}
Under the choices of \(a_i\) and \(C_i\) in \eqref{seq}, we rewrite this into the following:
\begin{equation} \label{step2-8}
\begin{aligned}
&-\frac{a_{i+1}}{u_{i+1}-u_{i+2}} \int_{\mathbb{T}^1}\Big(\int_{\x_{i+1}}^{\x^{i+2}} w\util' \,d\x\Big)^2 \,dy
-\frac{C_{i+2}}{u_{i+1}-u_{i+2}}\int_{\mathbb{T}^1} \Big(\int_{J_{i+2}} w\util'\,d\x\Big)^2 \,dy\\
&\qquad
\le -\frac{a_{i+2}}{u_{i+1}-u_{i+2}} \int_{\mathbb{T}^1}\Big(\int_{\x_{i+2}}^{\x_{i+1}}w\util' \,d\x\Big)^2 \,dy.
\end{aligned}
\end{equation}
Then, using Lemma \ref{KV} with Proposition \ref{prop:key} and \eqref{ztu0}, we obtain
\begin{equation} \label{step2-9}
\begin{aligned}
&-\frac{a_{i+2}}{u_{i+1}-u_{i+2}} \int_{\mathbb{T}^1}\Big(\int_{\x_{i+2}}^{\x_{i+1}}w\util' \,d\x\Big)^2 \,dy
-\frac{a_{i+2}}{2}(\lbar_{i+2})^{-1} \Dcal[\x_{i+2},\x_{i+1}]\\
&\qquad
\le -\frac{a_{i+2}}{u_{i+1}-u_{i+2}} \int_{\mathbb{T}^1}\Big(\int_{u_{i+2}}^{u_{i+1}}w \,dz\Big)^2\, dy
-\frac{a_{i+2}}{2} \int_{\mathbb{T}^1}\int_{u_{i+2}}^{u_{i+1}} (u_{i+1}-z)(z-u_{i+2}) (w_z)^2 \,dz \,dy\\
&\qquad
\le -a_{i+2} \int_{\mathbb{T}^1}\int_{u_{i+2}}^{u_{i+1}} w^2 \,dz \,dy
=-a_{i+2} \int_{\mathbb{T}^1}\int_{\x_{i+2}}^{\x_{i+1}} w^2 \util' \,d\x \,dy.
\end{aligned}
\end{equation}
This with \(\Hcal_{i+2}\) recovers the induction hypothesis \eqref{ind-hypo0} for \(i+2\)---gathering \eqref{step2-1}-\eqref{step2-9}, we have
\begin{equation} \label{step2}
\begin{aligned}
&-\Big(\frac{1}{2}+a_i\Big) \int_{\mathbb{T}^1}\int_{\x_i}^{\x_{i-1}} w^2 \util' \,d\x \,dy
+\Hcal_{i+1}+\Hcal_{i+2}\\
&-(0.81)(\r_*)^{-(i+1)}C_{i+1}\Dcal[\x_{i+1},\x^{i+1}]
-(0.82)(\r_*)^{-(i+2)}C_{i+2}\Dcal[\x_{i+2},\x^{i+2}]\\
&-\frac{1}{2}\Big(\frac{1}{2}+a_{i+1}\Big)(1.94)^{-1}(\r_*)^{-(i+1)} \Dcal[\x_{i+1},\x_i]
-\frac{a_{i+2}}{2}(0.51)(\r_*)^{-(i+2)} \Dcal[\x_{i+2},\x_{i+1}]\\
&\qquad\le -\Big(\frac{1}{2}+a_{i+2}\Big) \int_{\mathbb{T}^1}\int_{\x_{i+2}}^{\x_{i+1}} w^2 \util' \,d\x \,dy.
\end{aligned}
\end{equation}
We now examine how much of the diffusion term is used.
We recall \eqref{seq}: 
\[
a_i=\frac{1}{15}2^{-i}, \quad
a_{i+1}=\frac{1}{15}2^{-(i+1)}, \quad
a_{i+2}=\frac{1}{15}2^{-(i+2)}, \quad
C_{i+1}=a_i+\frac{3}{2}+\frac{1}{2a_i}, \quad
C_{i+2}=a_{i+1}.
\]
Note that for any odd \(i\in\NN\), \(C_{i+3}\le4C_{i+1}\), and thus we have
\[
(0.81)(\r_*)^{-(i+1)}C_{i+1}
\le (0.81)(\r_*)^{-2}C_2
< 0.63.
\]
Moreover, for any odd \(i\in\NN\), the other coefficients of the diffusion term are all less than \(1\):
\[
(0.82)(\r_*)^{-(i+2)}C_{i+2}, \quad
\frac{1}{2}\Big(\frac{1}{2}+a_{i+1}\Big)(1.94)^{-1}(\r_*)^{-(i+1)}, \quad
\frac{a_{i+2}}{2}(0.51)(\r_*)^{-(i+2)} < 0.01.
\]
Thus, from \eqref{step2}, we find that for each odd \(i\in\NN\),
\begin{equation} \label{step2-0}
\begin{aligned}
&-\Big(\frac{1}{2}+a_i\Big) \int_{\mathbb{T}^1}\int_{\x_i}^{\x_{i-1}} w^2 \util' \,d\x \,dy
+\Hcal_{i+1}+\Hcal_{i+2}
-(0.63)\Dcal[\x_{i+1},\x^{i+1}]\\
&-(0.01)\Dcal[\x_{i+2},\x^{i+2}]
-(0.01)\Dcal[\x_{i+1},\x_i]
-(0.01)\Dcal[\x_{i+2},\x_{i+1}]\\
&\qquad\le -\Big(\frac{1}{2}+a_{i+2}\Big) \int_{\mathbb{T}^1} \int_{\x_{i+2}}^{\x_{i+1}} w^2 \util' \,d\x \,dy.
\end{aligned}
\end{equation}

Combining the base case \eqref{step0and1} and the inductive step \eqref{step2-0}, we observe that for any interval on \((a,b)\), less than \(\frac{9}{10}\) of the diffusion term, i.e., \(\frac{9}{10}\Dcal[a,b]\) is used.
Thus, we conclude that
\begin{align*}
\frac{d}{dt}\frac{1}{2}\int_{\Omega} w^2 \,d\xi \,dy
&=-\frac{M}{2s} \int_{\mathbb{T}^1}\Big(\int_\RR w \util'\,d\x\Big)^2 \,dy
-\frac{1}{2}\int_{\mathbb{T}^1}\int_\RR w^2 \util' \,d\x \,dy
- \int_{\mathbb{T}^1} \int_\RR (w_\x)^2 \,d\x \,dy\\
&\le -\frac{1}{10}\int_{\mathbb{T}^1}\int_\RR (w_\x)^2 \,d\x \,dy,
\end{align*}
which yields \eqref{induction step}.
This completes the proof of Theorem \ref{induction step thm}.
\end{proof}
\
\\
{\bf Acknowledgement:} The first author is partially supported by National Science Foundation (DMS-2306258). The third author is partially supported by National Science Foundation (DMS-2206218).
\bibliographystyle{plain}
\bibliography{reference}
\end{document}

%% file: newmacros.tex
\usepackage{amssymb,mathrsfs,mathtools,graphicx,enumerate,color,colortbl}
\usepackage{hyperref}
\usepackage{comment}

\newcommand{\Dcal}{\mathcal{D}}

\newcommand{\Hcal}{\mathcal{H}}

\newcommand{\Scal}{\mathcal{S}}

\renewcommand{\hbar}{\bar{h}}

\newcommand{\lbar}{\bar{\lambda}}

\newcommand{\util}{\tilde{u}}

\renewcommand{\d}{\delta}

\renewcommand{\l}{\lambda}

\newcommand{\x}{\xi}

\renewcommand{\r}{\rho}
\newcommand{\s}{\sigma}

\newcommand{\e}{\varepsilon}

\newcommand{\RR}{{\mathbb R}}
\newcommand{\NN}{{\mathbb N}}

\newcommand{\abs}[1]{\left|#1\right|}

\newcommand{\norm}[1]{\left\|#1\right\|}

\newcommand{\step}[1]{\vskip0.2cm \noindent{\it Step #1:} }

\definecolor{black}{rgb}{0.0, 0.0, 0.0}
\definecolor{red}{rgb}{1.0, 0.5, 0.5}